\documentclass[11pt,A4]{amsart}
\usepackage[english]{babel}
\usepackage{amsmath,amsfonts,amssymb,amsthm,mathrsfs,bbm}
\usepackage[font=sf, labelfont={sf,bf}, margin=1cm]{caption}
\usepackage{graphicx,graphics}
\usepackage{epsfig}
\usepackage{latexsym}
\usepackage[applemac]{inputenc}
\usepackage{ae,aecompl}
\usepackage{pstricks}
\usepackage{enumerate}
\usepackage{xcolor}
\usepackage[pdfpagemode=UseNone,bookmarksopen=false,colorlinks=true,urlcolor=blue,citecolor=blue,citebordercolor=blue,linkcolor=blue]{hyperref}%\textwidth 173mm \textheight 235mm \topmargin -50pt \oddsidemargin
\usepackage[normalem]{ulem}

\usepackage[top=2.5cm,left=2.5cm,right=2.5cm,bottom=2.5cm]{geometry}
\usepackage{comment}

\usepackage[all, knot, cmtip]{xy}

\newcommand{\ndR}{\mathbb{R}}
\newcommand{\ndC}{\mathbb{C}}

\renewcommand{\Pr}[1]{\mathbb{P}(#1)}
\newcommand{\Prb}[1]{\mathbb{P}\left( #1 \right)}
\newcommand{\Ex}[1]{\mathbb{E}[#1]}
\newcommand{\Exb}[1]{\mathbb{E}\left[ #1 \right]}

\newcommand{\one}{\mathbbm{1}}

\newcommand{\cO}{\mathcal{O}}

\newcommand{\cL}{\mathcal{L}}

\newcommand{\cT}{\mathcal{T}}
\newcommand{\cR}{\mathcal{R}}

\newcommand{\cD}{\mathcal{D}}

\newcommand{\cW}{\mathcal{W}}

\newcommand{\mD}{\mathsf{D}}

\newcommand{\mO}{\mathsf{O}}

\newcommand{\mH}{\mathsf{H}}
\newcommand{\mT}{\mathsf{T}}

\newcommand{\me}{\mathsf{e}}

\newcommand{\CRT}{\mathcal{T}_\me}

\newcommand{\He}{\textnormal{H}}
\newcommand{\Di}{\textnormal{D}}

\newcommand{\he}{\text{h}}

\newcommand{\eqdist}{\,{\buildrel d \over =}\,}

\newcommand{\convdis}{\,{\buildrel d \over \longrightarrow}\,}

\newcommand{\Seq}{\textsc{SEQ}}

\newtheorem{theorem}{Theorem}[section]

\newtheorem{lemma}[theorem]{Lemma}

\newtheorem{remark}[theorem]{Remark}

\numberwithin{equation}{section}

\title{\textbf{Geometry of large Boltzmann outerplanar maps}}
\date{}

\author{Sigurdur \"Orn Stef\'ansson}
\address[Sigurdur \"Orn Stef\'ansson]{University of Iceland}
\email{sigurdur@hi.is}
\thanks{The first author acknowledges partial support from the University of Iceland Research Fund and is grateful for the hospitality at the University of Zurich. The second author thanks the German Research Foundation (fellowship STU 679/1-1) for their support.}

\author{Benedikt Stufler}
\address[Benedikt Stufler]{University of Zurich}
\email{benedikt.stufler@math.uzh.ch}

\begin{document}

	\maketitle
	
%	\let\thefootnote\relax\footnotetext{ \\\emph{MSC2010 subject classifications}. 60F17, 60C05 \\
%		\emph{Keywords and phrases.} scaling limits, stable loop trees, outerplanar maps}
	
%	\vspace {-0.5cm}

\begin{abstract}
	We study the phase diagram of random outerplanar maps sampled according to non-negative Boltzmann weights that are assigned to each face of a map. We prove that for certain choices of weights the map looks like a rescaled version of its boundary when its number of vertices tends to infinity. The Boltzmann outerplanar maps are then shown to converge in the Gromov--Hausdorff sense towards the $\alpha$-stable looptree introduced by Curien and Kortchemski~(2014), with the parameter $\alpha$ depending on the specific weight-sequence. This allows us to describe the transition of the asymptotic geometric shape from a deterministic circle to the Brownian tree.
\end{abstract}

\section{Introduction}
A planar map may be described as a proper embedding of a connected graph into the sphere, considered up to continuous deformations. In order to eliminate possible internal symmetries one usually distinguishes and orients a root edge. The probabilistic study of these objects has lead to a rich and beautiful theory, see for example the survey paper~\cite{MR3025391} and references given therein. The connected components of the complement of a planar map are its faces, and the unique face that lies to the right of the oriented edge is termed the outer face. This face is usually drawn as the unbounded face in plane representations. The number of edges adjacent to a face is its degree. A planar map is termed outerplanar, if all its vertices are adjacent to the outer face. 

Classes of outerplanar maps have received some attention in recent literature: Bonichon, Gavoille, and Hanusse~\cite{MR2185278} gave a combinatorial encoding of outerplanar maps in terms of certain bi-coloured plane trees. Combining this encoding with probabilistic techniques, Caraceni~\cite{caraceni2016} described the asymptotic geometric shape of uniform outerplanar maps, establishing Aldous' Brownian tree as their Gromov--Hausdorff scaling limit. The scaling limit and the asymptotic enumerative formula for outerplanar maps were later recovered (and extended to related classes such as bipartite outerplanar maps) in~\cite{stufler2017} by representing outerplanar maps as tree-like orderings of dissections of polygons. Geffner and Noy~\cite{MR3650252} established bijections between outerplanar maps and certain Dyck-paths with marked steps, providing precise enumerative expressions for various classes of outerplanar maps with respect to the number of edges and vertices.

Given a sequence of non-negative weights, we may assign to any planar map its Boltzmann weight given by the product of weights corresponding to the degrees of the (inner) faces. Rather than restricting ourselves to uniformly sampled maps from some class, we may sample maps of a given size parameter (for example, the number of edges, vertices, or faces) with probability proportional to their Boltzmann weight, and observe their behaviour as this size parameter tends to infinity. The study of such Boltzmann planar maps became quite popular~\cite{MR2778796, MR3342658,MR3484733,MR3646061,2016arXiv161208618M,2017arXiv170401950R}, and led to the discovery of many interesting phenomena that differ greatly from the uniform case.

Caraceni~\cite{caraceni2016} showed that for uniform outerplanar maps the inner faces typically have small degree, and the map has a tree-like geometric shape. A similar behaviour was observed later for further natural combinatorial classes of outerplanar maps~\cite{stufler2017}. But which interesting phenomena may be observed in weighted outerplanar maps? It was recently shown in~\cite{2016arXiv161202580S} that for arbitrary weight sequences the Boltzmann outerplanar map admits a local weak limit that describes the asymptotic behaviour near the root-edge, and a Benjamini--Schramm limit that describes the vicinity of a uniformly selected vertex. It was observed that there are three characteristic regimes, numbered I, II, and III. In the first, the limits have no doubly-infinite paths, and the shapes of the two local limits bear some similarity with Kesten's tree \cite{MR871905} (see also \cite[Sec. 5]{MR2908619}). In the second regime, both limits contain a marked doubly infinite path that corresponds to a face of macroscopic degree in the random outerplanar maps. In the third regime, the local behaviour near the fixed and random root is degenerate: both limits are equal to a single deterministic doubly infinite path. The regimes for the local limits provide hints on what to expect in the global limit, but a more fine-grained distinction is necessary. 

As our main result, we prove convergence towards Curien and Kortchemski's  \cite{MR3286462} $\alpha$-stable loop-tree $(\mathscr{L}_\alpha, d_{\mathscr{L}_\alpha})$ for $1< \alpha < 2$ in a certain sub-regime of case I.  We also discuss the boundary cases where the scaling limit is given by the  Brownian tree ($\alpha = 2$) in another subcase of case I, and by a deterministic circle ($\alpha = 1$) in "well-behaved" subcases of type II and III. The rough strategy is as follows. The outerplanar map looks like a tree-like arrangement of dissections of polygons that are glued together at the cutvertices of the map.  The dissections of polygons themselves also admit a tree-like combinatorial encoding. Hence the weight sequence for the faces of the outerplanar yields branching weights for the tree that controls the locations and sizes of the dissections, and uniform branching weights for the family of trees that control the arrangement of chords within the dissections. These two sequences of branching weights are related, and we carefully describe a general family of face-weights where the first associated sequence of branching weights lies in the universality class of the $\alpha$-stable tree (known from the works \cite{2015arXiv150404358K,MR3185928,MR1964956,MR3050512}), and the second sequence lies in the condensation regime (known from  \cite{MR2764126,MR3335012, MR2908619}). So any single large weighted dissection asymptotically looks like a circle whose circumference is about a constant fraction of its number of vertices, and the tree controlling the locations of the dissections looks like an $\alpha$-stable L\'evy tree. This enables us relate the metric on the Boltzmann outerplanar map to a rescaled version of the metric on its boundary. That is, our main observation is that asymptotically the map looks like a rescaled version of the frontier of the outer face. The scaling limit then attained by verifying convergence of the boundary of the outerplanar map towards the $\alpha$-stable loop-tree, which follows from a general result of Curien and Kortchemski \cite[Thm. 4.1]{MR3286462}. Here our arguments are analogous to a recent result by Richier~\cite{2017arXiv170401950R}, who proved a scaling limit for the boundary of Boltzmann planar maps that are not restricted to be outerplanar.

\section{Notation}

\label{sec:notation}

The interplay of face-weighted dissections and outerplanar maps plays a major role in this work. We recall a variety of relevant notions associated to these objects. Throughout, we fix a sequence $\iota = (\iota_k)_{k \ge 3}$ of non-negative weights such that at least one weight is positive.

\subsection{Dissections of polygons}

We let $\cD^\gamma$ denote the class (or collection) of $\iota$-face-weighted dissections of polygons, where one edge is distinguished and oriented (such that the outer face lies to its "right"), and the origin of the root-edge does not contribute to the total number of vertices. That is, an $n$-sized $\cD^\gamma$-object is given by an oriented $(n+1)$-gon with non-intersecting chords. The smallest such object has size $1$ and is given by a single oriented edge.

The $\gamma$-weight of a dissection $D$ is given by
\[
	\gamma(D) = \prod_{F} \iota_{|F|}
\]
with $F$ ranging over all inner faces of $D$, and $|F|$ denoting the degree of the face $F$. For the dissection consisting of a single link the product is over an empty set, and hence this link receives weight $1$.  We will let $\mD_n^\gamma$ denote a random $n$-sized $\cD^\gamma$-object drawn with probability proportional to its weight.

There are various parameters associated to the class of weighted dissections. Its generating series $\cD^\gamma(z)$ counts the sums of weights of dissections with a common size. That is, its $n$-th coefficient is given by
\begin{align*}
	[z^n]\cD^\gamma(z) = \sum_{D, |D| = n} \gamma(D)
\end{align*}
with the index $D$ ranging over all $\cD^\gamma$-objects with size $n$.

Face-weighted dissections are known to admit a tree-like encoding via the Ehrenborg--M\'endez isomorphism, see for example~\cite[Ch. 6.1.3]{2016arXiv161202580S} and Section~\ref{sec:comb} below for details. In terms of generating series, this is expressed by the recursive equation
\begin{align}
	\label{eq:diss}
			\cD^\gamma(z) = z \phi_\cD( \cD^\gamma(z))
\end{align}
	with 
	\[
		\phi_\cD(z) = 1/(1 - s(z)), \qquad s(z)  = \sum_{k \ge 1} \iota_{k+2}z^k.
	\]
	Recursive equations like \eqref{eq:diss} pertain to the study of simply generated trees, and the following notation is based on Janson's comprehensive survey \cite[Sec. 3]{MR2908619} on the subject.
	
	We let $\rho_\cD$ and $\rho_{\phi_\cD}$ denote the radii of convergence of $\cD^\gamma(z)$ and $\phi_\cD(z)$, and set
	\[
		\nu_\cD = \lim_{x \nearrow \rho_{\phi_\cD}} \psi_\cD(x), \qquad \psi_\cD(x) = x \phi_\cD'(x) / \phi_\cD(x).
	\]
	We also define a parameter $\tau_\cD$ in a manner that depends on $\nu_\cD$. If $\nu_\cD \ge 1$, we let it be the unique solution of the Equation $\psi_\cD(x) = 1$, and if $\nu_\cD < 1$ we set $\tau_\cD$ equal to $\rho_{\phi_{\cD}}$. That is,
	\[
		\tau_\cD = \begin{cases}
						\text{unique solution to } \psi_\cD(\tau_\cD) = 1, &\nu_\cD \ge 1 \\
						\rho_{\phi_{\cD}}, &\nu_\cD < 1.
					\end{cases}
	\] 
	We furthermore set
	\[
		\sigma_\cD^2 = \tau_\cD \psi_\cD'(\tau_\cD).
	\]
	It holds that \[\cD^\gamma(\rho_\cD) = \tau_\cD\] and \[\rho_\cD = \tau_\cD / \phi_\cD(\tau_\cD).\]
	See for example~\cite[Sec. 7]{MR2908619} for detailed justifications  given in a more general context.
	
\subsection{Face-weighted outerplanar maps}

We let $\cO^\omega$ denote the class of $\iota$-face-weighted outerplanar maps, where one edge that is incident to the outer face is distinguished and given a direction, such that the outer face lies to its "right". %Outerplanar maps are planar maps with the property that all its vertices lie on the frontier to the outer face.
 We will refer to the number of vertices of such an object $O$ as its size $|O|$.

We define the $\omega$-weight of an outerplanar map $O$ by the product of weights corresponding to its inner faces. That is,
\[
	\omega(O) = \prod_{F} \iota_{|F|},
\]
with $F$ ranging over all inner faces of $O$. We let $\mO_n^\omega$ denote a random $n$-sized outerplanar map drawn with probability proportional to its $\omega$-weight. 
We also define its generating series $\cO^\omega(z)$ by
\[
\cO^\omega(z) = \sum_{n \ge 1} z^n \sum_{O, |O|=n} \omega(O).
\]

Outerplanar maps admit a tree-like encoding based on the block-decomposition, see for example \cite{stufler2017}, \cite[Ch. 6.1.4]{2016arXiv161202580S}, and Section~\ref{sec:comb} below for details. This results in the recursive relation
	\[
		\cO^\omega(z) = z \phi_\cO(\cO^\omega(z))
	\]
	with 
	\[
		\phi_\cO(z) = 1/(1 - \cD^\gamma(z)).	
	\]
		We let $\rho_\cO$ and $\rho_{\phi_\cO}$ denote the radii of convergence of $\cO^\omega(z)$ and $\phi_\cO(z)$, and set
	\[
	\nu_\cO = \lim_{x \nearrow \rho_{\phi_\cO}} \psi_\cO(x), \qquad \psi_\cO(x) = x \phi_\cO'(x) / \phi_\cO(x).
	\]
	Similarly as for the dissections, we may define
	\[
	\tau_\cO = \begin{cases}
	\text{unique solution to } \psi_\cO(\tau_\cO) = 1, &\nu_\cO \ge 1 \\
	\rho_{\phi_{\cO}}, &\nu_\cO \le 1
	\end{cases}
	\]
	and
	\[
	\sigma_\cO^2 = \tau_\cO \psi_\cO'(\tau_\cO).
	\]
	As explained in detail in \cite[Sec. 7]{MR2908619} in a more general context, it holds that
	\[\cO^\omega(\rho_\cO) = \tau_\cO
	\]
	 and 
	 \[\rho_\cO = \tau_\cO / \phi_\cO(\tau_\cO).
	 \]

The parameters $\nu_\cO$ and $\nu_\cD$ are crucial in determining the behaviour of $\mD_n^\gamma$ and $\mO_n^\omega$. By \cite[Lem. 6.26]{2016arXiv161202580S} it holds that
\begin{align}
\label{eq:cases}
\nu_\cO = 	\begin{cases}
\infty, &\nu_\cD \ge 1 \\
\infty, &0 < \nu_\cD < 1, \rho_{\phi_\cD} \ge 1 \\
\frac{\tau_\cD}{(1 - \tau_\cD)(1 - \nu_\cD)}, &0 < \nu_\cD < 1, 0<\rho_{\phi_\cD}<1 \\
0, &\nu_\cD= 0,
\end{cases}
\end{align}
and those are the only possible cases. 

\subsection{Plane trees and looptrees} \label{ss:looptree}
A plane tree is a tree drawn in the plane with a marked oriented edge. We will call the origin of the root edge the root vertex. For each vertex $v$ we will call its neighbours which are further from the root vertex than $v$ its offspring and $v$ their parent. Offspring belonging to the same parent are referred to as siblings and due to the planar embedding of the tree it makes sense to speak of the order of the siblings from left to right. Another concept which will be used later is that of an ancestor. A  vertex $v$ is said to be an ancestor of a vertex $u$ if it lies on the unique path from $u$ to the root vertex.

A central concept in this paper is that of a looptree. A discrete looptree may be defined by starting from a plane tree $\mT$ and modifying the edge set as follows: Remove all edges from the tree. Each vertex which has offspring is then connected with an edge to its leftmost and its rightmost offspring and offspring of the same vertex are connected by an edge if they are adjacent as siblings, see Fig.~\ref{fi:loop}. We will denote the looptree associated to the tree $\mT$ by $\mathscr{L}(\mT)$.

In Theorem \ref{TEMAIN} we will prove weak convergence of rescaled random outerplanar maps towards the so--called $\alpha$--stable looptree which was introduced by Curien and Kortchemski. We will not give a formal definition of this random compact metric space here but refer to their paper \cite[Sec. 2]{MR3286462} for details. Informally, one may view it as the $\alpha$-stable tree \cite{MR1954248} in which every branch point of large degree is blown up into a circle. This is entirely analogous to the definition of a discrete looptree from a plane tree as given above but is technically more involved. 

\begin{figure}[t]
	\centering
	\begin{minipage}{1.0\textwidth}
		\centering
		\includegraphics[width=0.4\textwidth]{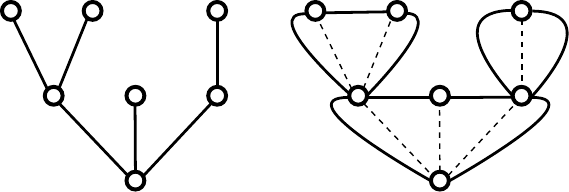}
		\caption{The looptree associated to a plane tree.}
		\label{fi:loop}
	\end{minipage}
\end{figure}

\section{The phase transition}

In order to describe the phase diagram of face-weighted outerplanar maps we are going to proceed systematically.  The key to this analysis are the parameters $\nu_\cO$ and $\nu_\cD$ introduced in Section~\ref{sec:notation}. Roughly said, $\nu_\cO$ determines how the location and sizes of the blocks (dissections) in the random outerplanar map $\mO_n^\omega$ behave asymptotically. The same goes for the parameter $\nu_\cD$ and the location and sizes of faces in the face-weighted dissection $\mD_n^\gamma$. 

We characterize three different regimes. In the \emph{circle regime} the shape of $\mO_n^\omega$ is completely determined by a giant $2$-connected block that exhibits a giant face, yielding a deterministic circle as limit after rescaling by roughly~$n$. In the \emph{Brownian tree regime} the blocks only stretch the geodesics by about a constant factor, yielding the Brownian continuum random tree as scaling limit after rescaling by about $\sqrt{n}$. Our main contribution is, however, in the \emph{$\alpha$-stable looptree regime} where the diameter of $\mO_n^\omega$ is shown to have roughly order $n^{1/\alpha}$ for $1 < \alpha < 2$. Here the location and sizes of the $2$-connected blocks (controlled by the parameter $\nu_\cO$) influence the scaling limit as well as the asymptotic shape of large Boltzmann dissections (determined by the parameter $\nu_\cD$). 

\subsection{The $\alpha$-stable loop-tree regime for $1<\alpha<2$}

In this section we focus on the case $\nu_\cO \ge 1$ and $\sigma_\cO = \infty$. This may only happen if $\nu_\cO = 1$, because $\nu_\cO>1$ implies that $\psi_\cO$ has radius of convergence larger than $\tau_\cO$, yielding $\sigma_\cO^2 = \tau_\cO \psi'(\tau_\cO) < \infty$.

We may characterize this setting. Let $r$ denote the radius of convergence of the face-weight generating series $s(z) = \sum_{k \ge 1} \iota_{k+2} z^k$. A proof of the following Lemma is given alongside all other proofs concerning the $\alpha$-stable loop-tree regime in Section~\ref{sec:palpha}.
\begin{lemma}
	\label{le:regime}
	\begin{enumerate}
		\item
	It holds that $\nu_\cO = 1$ if and only if  
		\[
		r,s(r) <1 \quad \text{and} \quad \frac{r}{1-r} + \frac{r s'(r)}{1 - s(r)} = 1.
		\]
	If this is the case, then $\tau_\cD = r$ and 
	\[
	\nu_\cD = \frac{r s'(r)}{1 - s(r)}<1.
	\]
		\item 	Suppose that $\nu_\cO = 1$. Then $\sigma_\cO = \infty$ if and only if  $s''(r) = \infty$.
	\end{enumerate}
\end{lemma}

We are going to consider weight-sequences of the form 
\begin{align}
\label{eq:weight}
\iota_{k+2} = L_k k^{-\alpha -1} r^{-k}
\end{align}
with $(L_k)_{k \ge 1}$ a slowly varying sequence and $\alpha,r$ positive constants. Slowly varying means that the sequence $L_k$ satisfies $L_{\lfloor tk \rfloor} / L_k \to 1$ as $k$ becomes large for any fixed $t >0$. See for example \cite[Sec. VIII.8]{MR0270403} for further details and standard results on slowly varying functions. Our main result is the following scaling limit.
\begin{theorem}
	\label{TEMAIN}
	Suppose that the face-weights $(\iota_k)_{k\geq3}$ satisfy Equation~\eqref{eq:weight} with $1<\alpha<2$ and $\nu_\cO = 1$. Then 
	\[
	\left( \mO_n^\omega, \left( \frac{nL_n\Gamma(-\alpha)}{1-s(r) }  \right)^{-1/\alpha} d_{\mO_n^\omega} \right) \convdis (\mathscr{L}_\alpha, d_{\mathscr{L}_\alpha})
	\]
	in the Gromov--Hausdorff sense.
\end{theorem}

Here $(\mathscr{L}_\alpha, d_{\mathscr{L}_\alpha})$ denotes  the $\alpha$-stable loop-tree constructed by Curien and Kortchemski~\cite{MR3286462} which we briefly described in Subsection \ref{ss:looptree}. The Gromov--Hausdorff distance $d_{\textsc{GH}}( (X,d_X), (Y,d_Y))$ between two compact metric spaces $(X,d_X)$ and $(Y, d_Y)$ is defined as the infimum of Hausdorff distances of isometric embeddings into any possible common metric space. Equivalently, it may be defined as the infimum \[
d_{\textsc{GH}}( (X,d_X), (Y,d_Y)) = \frac{1}{2}\inf_{R} \mathrm{dis}(R)
\] of distortions $\mathrm{dis}(R)$ of all  correspondences $R$ between the two spaces. Here correspondence means that $R \subset X \times Y$ is a relation where to any $x\in X$ corresponds at least one $y \in Y$ and vice versa. The distortion of such an object is defined by \[\mathrm{dis}(R) = \sup_{(x_1, y_2), (x_2,y_2) \in R}|d_X(x_1, x_2) - d_Y(y_1,y_2)|.\] We refer the reader to~\cite[Sec. 6]{MR2571957}  for details.

\begin{remark}
	\label{re:weight}
	The requirements of Theorem~\ref{TEMAIN} are satisfiable. For each $\alpha>1$ there are constants $c,r>0$ such that the weight-sequence 
	\[
		\iota_{k+2} = c k^{-\alpha-1} r^{-k}, \quad k \ge 1
	\]
	satisfies $\nu_\cO=1$. In this setting, $\sigma_\cO = \infty$ holds precisely for $\alpha \le 2$.
	To see this, let $\zeta(s) = \sum_{k \ge 1} k^{-s}$ denote the Riemann zeta function. In order for the requirements of Lemma~\ref{le:regime} to be fulfilled, we may choose $c>0$ sufficiently small such that
	\begin{align}
	\label{eq:weight1}
	c \zeta(\alpha+1) < 1 \quad \text{and} \quad \frac{c \zeta(\alpha)}{1 - c\zeta(\alpha+1)} < 1.
	\end{align}
	There is a constant $c_0(\alpha) >0$ such that this holds precisely for all $0 < c < c_0(\alpha)$. Finally, we let $0 < r < 1/2$ denote the unique parameter with
	\begin{align}
	\label{eq:weight2}
	\frac{r}{1-r} = 1 - \frac{c \zeta(\alpha)}{1 - c\zeta(\alpha+1)}.
	\end{align}

	More generally, for any
	$1 < \alpha < 2$  and any slowly varying sequence $(L_k)_k$ one may always find an $0<r<1$ and a slowly  varying sequence $(L'_k)_k$ which is asymptotically equivalent to $L_k$ up to multiplication by a constant such that the weights $L'_k k^{-\alpha-1}r^{-k}$ satisfy the requirements of Theorem~\ref{TEMAIN}.
\end{remark}

In the proof of Theorem~\ref{TEMAIN} we consider the boundary $\bar{\mO}_n^\omega$ of the map $\mO_n^\omega$. That is, $\bar{\mO}_n^\omega$ is the map obtained from $\mO_n^\omega$ by deleting all edges that do not lie on the frontier of the outer face. Our main observation is that the distortion between the space $(\mO_n^\omega, d_{\mO_n^\omega})$ and the contracted boundary $( (\bar{\mO}_n^\omega, (1-\nu_\cD)d_{\bar{\mO}_n^\omega})$ lies in $o_p( (nL_n)^{-1/\alpha})$. The scaling limit for $\mO_n^\omega$ then follows by a limit for $\bar{\mO}_n^\omega$, which converges towards the $\alpha$-stable looptree after rescaling the metric by $b_n^{-1}$ with 
\[
b_n = \left( \frac{nL_n\Gamma(-\alpha)}{1-s(r) }  \right)^{1/\alpha} \frac{1-r}{r}.
\]
Thus the scaling factor in Theorem~\ref{TEMAIN} consists of the product of the two factors $1-\nu_\cD = r/(1-r)$ and~$b_n^{-1}$.

In order to get convergence of the boundary, we argue similarly as Richier \cite{2017arXiv170401950R}, who gave such a scaling limit for the boundary of Boltzmann planar maps that are not required to be outerplanar. That is, we approximate $\bar{\mO}_n^\omega$ by a  discrete looptree that shares the same set of vertices and is associated to a critical Galton--Watson tree conditioned on having many leaves whose offspring distribution lies in the domain of attraction of an $\alpha$-stable law. The limit then follows by scaling limits for discrete looptrees by  Curien and Kortchemski~\cite{MR3286462}.

Actually, our arguments even show that if we equip all discrete spaces with the uniform measure on their points, then the Gromov--Hausdorff--Prokhorov distance between the outerplanar map and the discrete looptree is negligible, since we only use approximation arguments with spaces sharing the same vertex set. So extending the scaling limit of Curien and Kortchemski \cite[Thm. 4.1]{MR3286462} to the Gromov--Hausdorff--Prokhorov distance would automatically entail a corresponding strengthening of Theorem~\ref{TEMAIN}.

\subsection{The circle regime}
\label{sec:circle}

We are interested in the setting $0<\nu_\cO < 1$. Letting $r$ denote the radius of convergence of the series $s(z) = \sum_{k \ge 1} \iota_{k+2} z^k$, we may characterize this setting as follows.
\begin{lemma}
	\label{le:reg2}
		It holds that $\nu_\cO < 1$ if and only if  
		\[
		r,s(r) <1 \quad \text{and} \quad \frac{r}{1-r} + \frac{r s'(r)}{1 - s(r)} < 1.
		\]
		If this is the case, then $\tau_\cD = r$,
		\[
		\nu_\cD = \frac{r s'(r)}{1 - s(r)}<1, \quad \text{and} \quad \nu_\cO = \frac{r}{(1-r)(1- \nu_\cD)}.
		\]
\end{lemma}
The proof is entirely analogous to the proof of the first claim of Lemma~\ref{le:regime}. We impose a regularity assumption, focusing on weights of the form
\begin{align}
\label{eq:we}
\iota_{k+2} = L_k k^{-\alpha -1} r^{-k}
\end{align}
with $(L_k)_{k \ge 1}$ a slowly varying sequence and $\alpha,r$ positive constants.

\begin{theorem}
		\label{te:circ}
		Suppose that the face-weights $(\iota_k)_{k\geq3}$ satisfy Equation~\eqref{eq:we} with $\alpha>1$ and $0<\nu_\cO < 1$. Let $\mu_n$ denote a uniformly at random chosen point of $\mO_n^\omega$, and $\mu$ a uniformly selected point on the circle  $C^1 = \{z \in \ndC \mid |z|=(2\pi)^{-1}\}$ of unit circumference. Then
		\[
		 \left(\mO_n^\omega, \frac{1-r}{nr} d_{\mO_n^\omega}, \mu_n \right) \convdis (C^1, d_{C^1}, \mu) 
		\]
		in the Gromov--Hausdorff--Prokhorov sense.
\end{theorem}

The Gromov--Hausdorff--Prokhorov distance $d_{\textsc{GHP}}( (X, d_X, \mu_X), (Y, d_Y, \mu_Y))$ between two compact metric spaces $(X, d_X)$ and $(Y, d_Y)$ equipped with Borel probability measures $\mu_X, \mu_Y$ may be defined as the infimum
\[
d_{\textsc{GHP}}( (X, d_X, \mu_X), (Y, d_Y, \mu_Y)) = \inf_{E, \varphi_X, \varphi_Y} \max(d_{\textsc{H}}(\varphi_X(X), \varphi_Y(Y)), d_{\textsc{P}}(\mu_X\phi_X^{-1}, \mu_Y\phi_Y^{-1} ) )
\]
with the index ranging over all isometric embeddings $\varphi_X: X \to E$ and $\varphi_Y: Y \to E$ into any common metric space $E$. Here $d_{\textsc{H}}(\varphi_X(X), \varphi_Y(Y))$ denotes the Hausdorff distance of the images of $X$ and $Y$, and $d_{\textsc{P}}(\mu_X\varphi_X^{-1}, \mu_Y\varphi_Y^{-1} )$ denotes the Prokhorov distances of the push-forwards of the measures $\mu_X$ and $\mu_Y$ along $\varphi_X$ and $\varphi_Y$. We refer the reader to~\cite[Sec. 6]{MR2571957} for details.

 The idea of Theorem~\ref{te:circ} is that the geometric shape of $\mO_n^\omega$ will be determined by a giant $2$-connected block, whose size is about $(1-\nu_\cO)n$. Hence the scaling limit of $\mO_n^\omega$ will be the same as the limit of the dissection $\mD_n^\gamma$ stretched by $1/ (1 -\nu_\cO)$. A priori, various qualitatively different continuum limits are known for Boltzmann dissections, and $\nu_\cD$ tells us which to expect. However, Equation~\eqref{eq:cases} guarantees that if $\nu_\cO<1$ then we also have $\nu_\cD<1$. So, by the same arguments, the geometric shape of $\mD_n^\gamma$ in this regime is determined by a giant face whose degrees is roughly $(1- \nu_\cD)n$. In total, $\mO_n^\omega$ looks like a circle with circumference $n(1-\nu_\cO)(1 - \nu_\cD)n = nr / (1-r)$. Compare also with a similar result for the boundary of a percolation cluster in the uniform random triangulation given by Curien and Kortchemski~\cite[Thm. 1.2]{MR3405619}.

\begin{remark}
	In the case $\nu_\cO = 0$ we expect that at least for certain choices of weights like $\iota_k \sim (k!)^\beta$ with $\beta >0$ the rescaled map $(\mO_n^\omega, n^{-1} d_{\mO_n^\omega})$ converges by similar arguments towards the circle $C^1$ of unit circumference. This is based on estimates on the sizes of fringe subtrees dangling from the root in simply generated trees in the super-condensation regime, see~\cite[Thm. 2.5]{MR2860856} and compare with~\cite[Prop. 3.5]{MR3342658}.
\end{remark}

\subsection{The Brownian tree regime}
\label{sec:btree}
The Brownian tree $(\CRT, d_{\CRT})$ was introduced by Aldous in his pioneering papers~\cite{MR1085326,MR1166406,MR1207226}. If $\nu_\cO \ge 1$ and $\sigma_\cO < \infty$, we would expect that
\begin{align}
\label{eq:crtlim}
(\mO_n^\omega, c_\omega n^{-1/2} d_{\mO_n^\omega}) \convdis (\CRT, d_{\CRT})
\end{align}
in the Gromov--Hausdorff sense for some constant $c_\omega$ that depends on the $\omega$-weights. This has been verified first by Caraceni~\cite{caraceni2016} for uniform outerplanar maps, and later in the more general subcase $\nu_\cO > 1$ \cite[Thm. 6.60]{2016arXiv161202580S}. However,  the arguments used there require certain random variables to have finite exponential moments, which is no longer the case if~$\nu_\cO = 1$.

It is reasonable to expect that \eqref{eq:crtlim} still holds if $\nu_\cO = 1$ and $\sigma_\cO < \infty$. (See also \cite[Sec. 5.3]{MR3382675}.) This is supported  by the fact \cite[Lem. 6.6.1]{2016arXiv161202580S} that in this setting the diameter $\Di(\mO_n^\omega)$ of $\mO_n^\omega$ has a stochastic lower bound of order  $n^{1/2}$ and satisfies the tail-bound
\[
\Pr{\Di(\mO_n^\omega) \ge x} \le C \exp(-c x^2/n)
\]
for all $n \ge 1$ and $x \ge 0$. This yields tightness of $(\mO_n^\omega, n^{-1/2}d_{\mO_n^\omega})$ and hence weak convergence along subsequences, but one would still have to verify uniqueness and properties of that limit.

Note that we may construct concrete weight-sequences that lie in the Brownian tree regime: The uniform case where $\iota_k = 1$ for all $k \ge 3$ is known to belong to the case $\nu_\cO = \infty > 1$. The case $\nu_\cO=1$ and $\sigma_\cO < \infty$ may be attained by defining the weight sequence $\iota$ according to Remark~\ref{re:weight} for $\alpha>2$.% Equations~\eqref{eq:weight}-\eqref{eq:weight2} but for any fixed $\beta > 3$ instead of $2 < \beta < 3$.

It is also natural to wonder what happens if $\sigma_\cO=\infty$ and the branching law of the simply generated tree in Lemma~\ref{le:couplingo1} below lies in the domain of attraction of a normal law, but we do not aim to pursue this question here.

\section{On tree-like combinatorial representations}
\label{sec:comb}

Before starting with the proofs of our main results, we discuss the tree-like structures that may be used to encode and sample face-weighted dissections and outerplanar maps. These structures may be encoded in a unified way by so called enriched trees and Schr\"oder enriched parenthesizations. 
Roughly speaking, given a combinatorial class $\cR$ an \emph{$\cR$-enriched plane tree} is a pair $(T, \beta)$ of a plane tree $T$ and a function $\beta$ that assigns to each vertex $v \in T$ a $d_T^+(v)$-sized $\cR$-structure $\beta(v)$.  By certain general principles these representations entail couplings of the random maps under consideration with simply generated trees indexed by their number of vertices or leaves. We briefly recall relevant parts of this theory following  \cite[Sec. 6.1.3, Sec. 6.1.4]{2016arXiv161202580S}, and refer the reader to this source for further details and background information.% for explicit details and the Boltzmann sampler framework developed in~\cite{MR2810913, MR2498128, MR2095975})

\subsection{Dissections}

\begin{figure}[t]
	\centering
	\begin{minipage}{1.0\textwidth}
		\centering
		\includegraphics[width=1.0\textwidth]{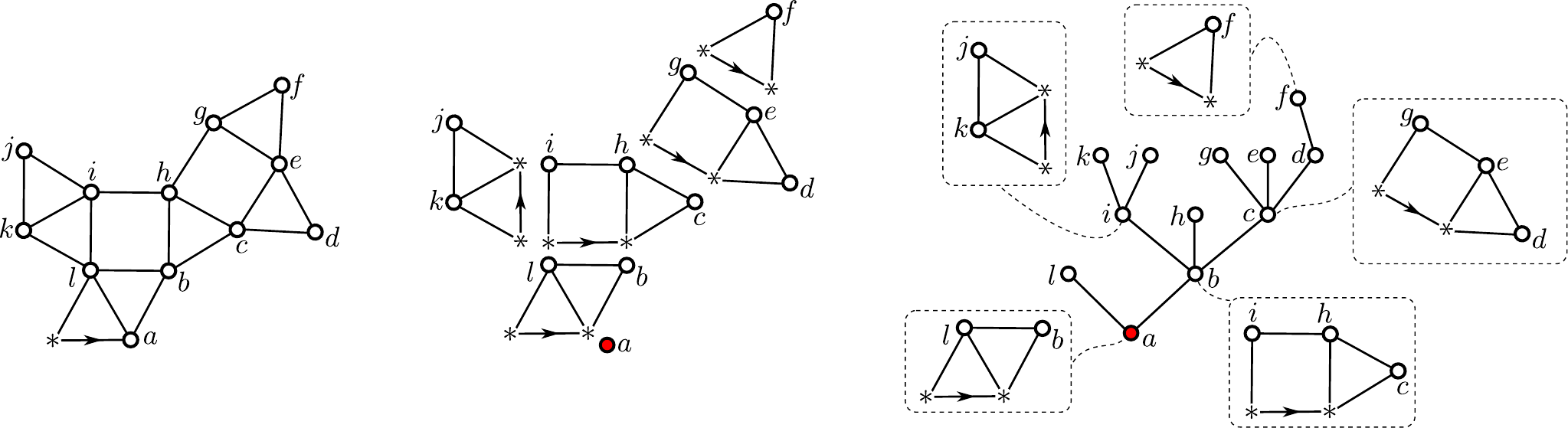}
		\caption{Decomposition of dissections of polygons into chord-restricted components.}
		\label{fi:decdis}
	\end{minipage}
\end{figure}

We consider face-weighted dissections having a counter-clockwise oriented root-edge on the boundary whose origin does not contribute to the total size. In general we refer to vertices that do not contribute to the total size as $*$-placeholder vertices. Any dissection  may be decomposed in a tree-like fashion into \emph{chord restricted} components as illustrated in the middle part of Figure~\ref{fi:decdis}. Here the term chord restricted dissection refers to a face-weighted dissections where both the origin and destination of the root-edge do not count as regular vertices and hence do not contribute to the total size, and where each chord must be incident to the destination of the root-edge.  The idea of the decomposition is to start with the maximal chord-restricted sub-dissection containing the root-edge, and then recursively continue in this manner with the ordered list of dissections attached to its boundary. 

Having decomposed a dissection $D$ into its components, we may form a tree $T$ as illustrated on the right hand side of Figure~\ref{fi:decdis} such that its vertices correspond bijectively to the non-$*$-vertices of~$D$.  The idea is that the root  of $T$ is given by the destination of the root-edge of $D$, and the offspring of the root is given by the non-$*$-placeholder vertices of the unique chord-restricted component $C$ containing the root-edge of $D$. Each non-$*$-vertex $v$ of $C$ corresponds to the edge that lies counter-clockwise next to it on the boundary of $C$, and hence also corresponds to the dissection $D(v)$ attached to $C$ via this edge. The tree $T$ is then constructed in a recursive manner such that the offspring of $v$ corresponds to the non-$*$-vertices of the chord-restricted component of $D(v)$ that contains the root-edge of $D(v)$ and so on. The tree $T$ does not capture all informations necessary to reconstruct the dissection $D$ from it. For this reason, we remember for each vertex $u \in T$ the unique chord-restricted component $\beta(u)$ of $D$ whose non-$*$-vertices correspond to the offspring set of $u$ in $T$. Thus the pair $(T, \beta)$ is a tree enriched with chord-restricted dissections from which we may reconstruct the corresponding dissection $D$.

Any chord-restricted dissection is  uniquely determined by the ordered list of its face-degrees in some canonical order. By general enumerative principles it follows that the sum of weights of chord restricted dissections with total size $n$ is given by the $n$th coefficient of the generating series $\phi_\cD(z) = 1 / (1- s(z))$. Let $\boldsymbol{\gamma} = (\gamma_k)_{k \ge 0}$ be the weight sequence corresponding to the coefficients $\gamma_k = [z^k] \phi_\cD(z)$. We may view $\boldsymbol{\gamma}$ as a branching weight sequence.  By \cite[Lem. 6.1]{2016arXiv161202580S} the random face-weighted dissection $\mD_n^\gamma$ corresponds to a simply generated tree decorated with conditionally independently chosen chord-restricted decorations:

\begin{lemma}[Dissections and simply generated trees]
	\label{le:dcoup}
	The dissection corresponding to the  following random enriched tree $(\cT_n^\cD, \beta_n^\cD)$ is distributed like the random face-weighted dissection $\mD_n^\gamma$.
	\begin{enumerate}
		\item Let $\cT_n^\cD$ be an $n$-vertex simply generated tree with weight-sequence $\boldsymbol{\gamma}$. 
		\item For each vertex $v$ of $\cT_n^\cD$ let $\beta_n^\cD(v)$ be drawn with probability proportional to its weight among all $d^+_{\cT_n^\cO}(v)$-sized chord restricted dissections.
	\end{enumerate}
\end{lemma}
If $\nu_\cD>0$ then $\cT_n^\cD$ is distributed like a Galton--Watson tree $\cT^\cD$ conditioned on having $n$ vertices, with the offspring distribution having probability generating function $\phi_\cD(\tau_\cD z) / \phi_\cD(\tau_\cD)$ \cite[Sec. 2, Sec. 4, Rem. 7.9]{MR2908619}.

As we stated above, chord restricted dissections correspond to ordered sequences of faces, which is reflected in the fact that their generating function $\phi_\cD(z)$ is the composition of the series $1/(1-z)$ associated with ordered sequences and the generating function $s(z)$ for the faces. Hence drawing a chord restricted dissection of a given size with probability proportional to its weight is an example of a Gibbs partition, a term phrased by Pitman~\cite{MR2245368} in his survey on combinatorial stochastic processes. This will allow us to apply results for this general model of random partitions later on in the proof of our main theorems.

\subsection{Outerplanar maps} \label{ss:optrees}

\begin{figure}[t]
	\centering
	\begin{minipage}{1.0\textwidth}
		\centering
		\includegraphics[width=1.0\textwidth]{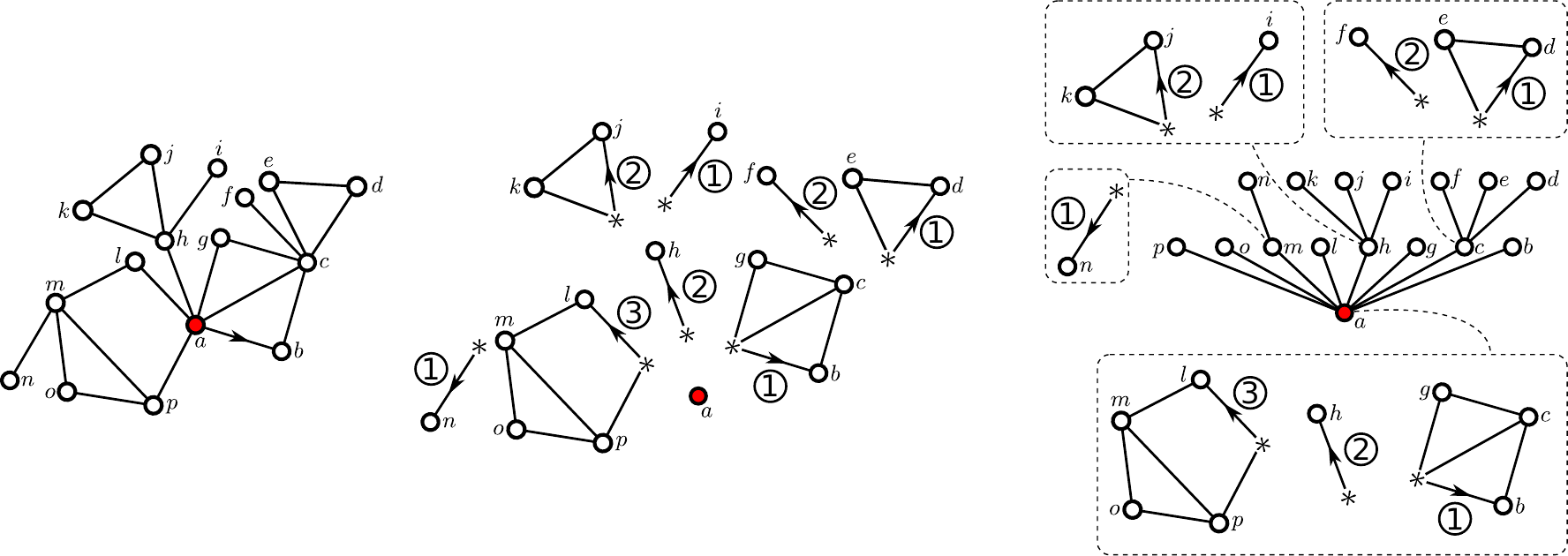}
		\caption{Enriched tree decomposition of outerplanar maps.}
		\label{fi:outervertex}
	\end{minipage}
\end{figure}

As illustrated in Figure~\ref{fi:outervertex}, outerplanar maps may be encoded in terms of trees enriched with ordered sequences of dissections. The vertices of the map correspond bijectively with the vertices of the tree. The idea is that outerplanar maps consist of collections of dissections glued together, and the precise information on which vertices should be glued together may be encoded by a plane tree. Specifically,  given a  tree enriched with ordered sequences of dissections $(T,\beta)$  such as on the right side of Figure~\ref{fi:outervertex}, we may form the corresponding outerplanar map as follows. We start with the root $v$ of $T$ and glue the ordered sequence of dissections $\beta(v)$ together in a counter-clockwise way at at their respective root-vertices. The root-edge of the first dissection in the sequence becomes the root-edge of the resulting map $S_v$. We then proceed recursively to form the outerplanar maps corresponding to the enriched fringe subtrees dangling from the root $v$ in $(T, \beta)$ and glue their root-vertices (that is, the origin of the root-edge) to the corresponding vertices of $S_v$.

We may define the weight of a sequence of dissections as the product of their individual weights, and its size as the number of non-$*$-vertices.
It follows from general enumerative principles that the sum of weights of ordered sequences of face-weighted dissections with size $k$ is given by $\omega_k = [z^k] \phi_\cO(z)$. We may interpret $\boldsymbol{\omega} = (\omega_k)_{k \ge 0}$ as a branching weight sequence. By \cite[Lem. 6.1]{2016arXiv161202580S} the random face-weighted outerplanar map $\mO_n^\omega$ corresponds to a simply generated tree with branching weights $\boldsymbol{\omega}$ decorated by ordered sequences of dissections:

\begin{lemma}[Outerplanar maps and simply generated trees]
	\label{le:couplingo1}
	The outerplanar map corresponding to the  following random enriched tree $(\cT_n^\cO, \beta_n^\cO)$ is distributed like~$\mO_n^\omega$.
	\begin{enumerate}
		\item Let $\cT_n^\cO$ be an $n$-vertex simply generated tree with weight-sequence $\boldsymbol{\omega}$. 
		\item For each vertex $v$ of $\cT_n^\cO$ let $\beta_n^\cO(v)$ be drawn with probability proportional to its weight among all $d^+_{\cT_n^\cO}(v)$-sized ordered sequences of dissections.
	\end{enumerate}
\end{lemma}

We emphasize that drawing an ordered sequence of dissections of a fixed size with probability proportional to its weight is an example of a Gibbs partition~\cite{MR2245368}.

As an alternative to the encoding of outerplanar maps in terms of decorated trees where the vertices of the map correspond to the vertices of the tree, there is an encoding in terms of trees enriched with dissections (this time counting the origin of the root-edge as a regular vertex) such that the vertices of the map correspond bijectively to the leaves of the tree. As the size index is given by the number of leaves, these enriched trees are also called Schr\"oder enriched parenthesizations.

\begin{figure}[t]
	\centering
	\begin{minipage}{1.0\textwidth}
		\centering
		\includegraphics[width=0.7\textwidth]{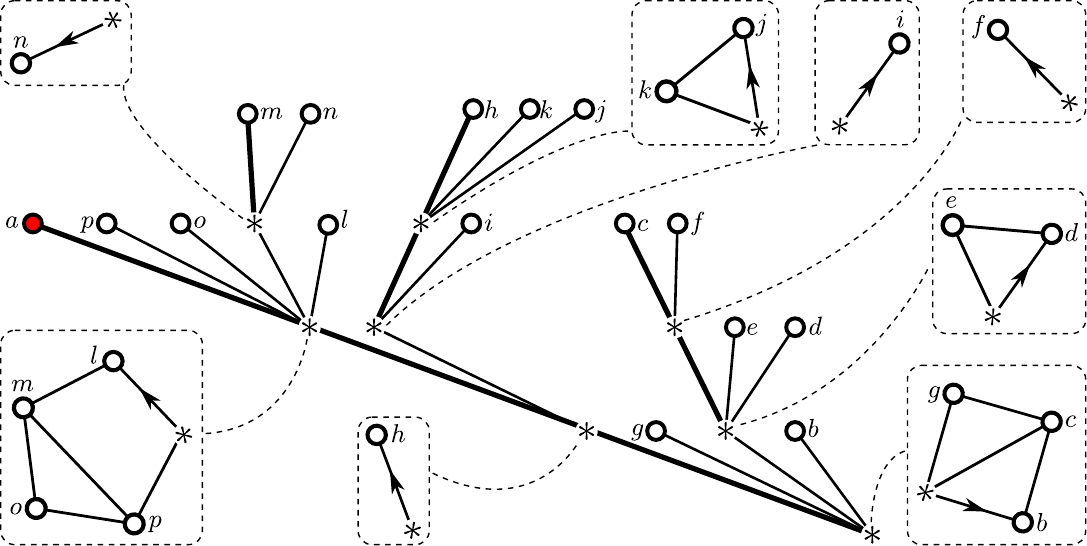}
		\caption{The enriched Schr\"oder parenthesizations representation of the outerplanar map displayed in Figure~\ref{fi:outervertex}. }
		\label{fi:outerleaf}
	\end{minipage}
\end{figure}  
 
The idea of this decomposition is that any outerplanar map is either a single vertex or an edge-rooted dissection of a polygon, where each vertex (including this time the origin of the root-edge) gets identified with the origin of the root-edge of another outerplanar map. The tree corresponding to an outerplanar  map $O$ is formed by starting with some root-vertex and adding offspring according to the number of vertices (including the origin of the root-edge) of the unique block $D$ (that is, maximal dissection) containing the root-edge of $O$. We may then proceed in this manner for the outerplanar maps attached to the boundary of $D$ in let's say clockwise order starting at the root-vertex, resulting in a plane tree $T$.  Note that the tree $T$ contains information on the sizes and locations of the dissection components of the map $O$, but not on the chords within the dissections. For this reason, we assign to each vertex $v \in T$ the  dissection $\delta(v)$ whose vertices correspond to the offspring vertex of $v$. Thus the outerplanar map $O$ may be reconstructed from $(T, \delta)$. See Figure~\ref{fi:outerleaf} for an illustration of the enriched parenthesization corresponding to the map we studied in Figure~\ref{fi:outervertex}.

We may use this decomposition to construct a coupling with a simply generated tree whose atoms are leaves. That is, the tree with ``size'' $n$ gets drawn with probability proportional to the product of weights assigned to its outdegrees among all plane trees with $n$ leaves. By \cite[Lem. 6.7]{2016arXiv161202580S} we may make use of the following sampling procedure.

\begin{lemma}[Outerplanar maps and simply generated trees whose atoms are leaves]
	\label{le:outleaf}
	The outerplanar map corresponding of the following enriched tree $(\tau_n^\cO, \delta_n^\cO)$ is distributed like~$\mO_n^\omega$.
	\begin{enumerate}
		\item Let $\tau_n^\cO$ be an $n$-leaf simply generated tree where each inner vertex with outdegree $d$ receives the weight $p_d := [z^{d-1}]\cD^\gamma(z)$ and each leaf receives weight $p_0 := 1$. 
		\item For each vertex $v$ of $\tau_n^\cO$ draw a dissection $\delta_n^\cO(v)$  of an $n$-gon with probability proportional to its weight. 
	\end{enumerate}
\end{lemma}
If $\nu_\cO >0$ then the tree $\tau_n^\cO$ is distributed like a Galton--Watson tree conditioned on having $n$ leaves, with offspring distribution having probability generating function $1 - \cD^\gamma(\tau_\cO) + z \cD^\gamma(z \tau_\cO)$.

 This follows by similar tilting and normalizing arguments (see \cite[Sec. 6.4]{2016arXiv161202580S} for some details) as for the standard case of simply generated trees with vertices as atoms. Essentially, for any $t>0$ with $\sum_{d \ge 1} p_d t^{d-1}  < 1$ there is a unique $a>0$ such that the tilted sequence $p_0(t) := a$ and $p_k(t) := p_kt^{k-1}$ is a probability weight sequence. The two sequences are equivalent in the sense that the $n$-leaf simply generated trees corresponding to them are identically distributed for all $n$. Hence $\tau_n^\cO$ is distributed like a Galton--Watson tree with branching law $(p_k(t))_{k \ge 0}$ conditioned on having $n$ leaves. The expected value of the tilted sequence is given by $\mu_t = \sum_{k \ge 1} k p_k t^{k-1}$. The parameter $t=\tau_\cO$ is an admissible choice of parameter, and the corresponding mean value $\mu_{\tau_\cO}$  satisfies $\mu_{\tau_\cO} = 1$ if and only if $\nu_\cO \ge 1$.

%\newpage

\section{Proofs in the $\alpha$-stable loop-tree regime}

\label{sec:palpha}

\subsection{Proof of Lemma~\ref{le:regime}}

We prove the two claims  separately.

\begin{proof}[First claim]
	By Equation~\eqref{eq:cases}, we know that $\nu_\cO = 1$ is equivalent to
	\begin{align}
	\label{eq:cond1}
	\frac{\tau_\cD}{(1 - \tau_\cD)(1 - \nu_\cD)} = 1,\quad 0 < \nu_\cD < 1,\quad 0<\rho_{\phi_\cD}<1.
	\end{align}

	Suppose that \eqref{eq:cond1} holds. In particular this entails $0 < \nu_\cD < 1$, and hence $\rho_{\phi_\cD} = \tau_\cD$. As $\phi_\cD(z) = (1 - s(z))^{-1}$ it follows that $0 < r = \tau_\cD <1$ and $0<s(r) < 1$. Hence
	\[
	\nu_\cD = \frac{r s'(r)}{1 - s(r)}.
	\]
	Inserting this expression for $\nu_\cD$ into the first equation of \eqref{eq:cond1} yields
	\[
	\frac{r}{1-r} + \frac{r s'(r)}{1 - s(r)} = 1,
	\]
	verifying one direction of the proof.
	
	As for the other direction, suppose that 
	\begin{align}
	\label{eq:cond2}
	\frac{r}{1-r} + \frac{r s'(r)}{1 - s(r)} = 1, \quad s(r) < 1, \quad r < 1.
	\end{align}
	Then $\rho_{\phi_{\cD}} = r$ and
	\[
	\nu_\cD = \frac{r s'(r)}{1 - s(r)} < 1.
	\]
	Thus $\tau_\cD = \rho_{\phi_\cD} =r$. The first equation of \eqref{eq:cond2} now reads
	\[
	\frac{\tau_\cD}{1 - \tau_\cD} + \nu_\cD = 1.
	\]
	This completes the proof.
\end{proof}

\begin{proof}[Second claim]
	As $\nu_\cO=1$, it holds that $\tau_\cO = \rho_{\phi_\cO}$. Since $\phi_\cO(z) = (1 - \cD^\gamma(z))^{-1}$, it follows that 
	\begin{align}
	\label{eq:sp}
	\tau_\cO = \rho_\cD  < \infty.
	\end{align} Thus, $\sigma_\cO^2 = \tau_\cO \psi'_\cO(\tau_\cO)$ is infinite if and only if $\psi'_\cO(\rho_\cD) = \infty$.
	
	We know that
	\[
	\psi_\cO(z) = z (\cD^\gamma)'(z) / (1 - \cD^\gamma(z))
	\]
	and hence
	\[
	\psi_\cO'(z) = \frac{ ((\cD^\gamma)'(z) + z(\cD^\gamma)''(z))(1 - \cD^\gamma(z)) + z (\cD^\gamma)'(z)^2}{(1 - \cD^\gamma(z))^2}
	\]
	Note that $\cD^\gamma(\rho_\cD) = \tau_\cD = r < 1$ since $\nu_\cO = 1$. Hence $\psi_\cO'(\rho_\cD)$ is infinite if and only if $(\cD^\gamma)'(\rho_\cD) = \infty$ or $(\cD^\gamma)''(\rho_\cD) =  \infty$.
	
	The recursive equation $\cD^\gamma(z) = z \phi_\cD(\cD^\gamma(z))$ implies
	\[
	(\cD^\gamma)'(z) = \frac{\phi_\cD(\cD^\gamma(z))}{1 - z \phi_\cD'(\cD^\gamma(z))}.
	\]
	The denominator is finite at $z = \rho_\cD$, as \[
	\rho_\cD \phi'_\cD(\cD^\gamma(\rho_\cD)) = \tau_\cD \phi'_\cD(\tau_\cD) / \phi_\cD(\tau_\cD) = \nu_\cD < 1.
	\]
	Hence 
	\[
	(\cD^\gamma)'(\rho_\cD) = \frac{\phi_\cD(\tau_\cD)}{1 - \nu_\cD} < \infty.
	\]
	
	Thus we have established that $\sigma_\cO = \infty$ if and only if $(\cD^\gamma)''(\rho_\cD) = \infty$. Differentiating the recursive equation $\cD^\gamma(z) = z \phi_\cD(\cD^\gamma(z))$ twice yields
	\[
	(\cD^\gamma)''(z) = 2 \phi_\cD'(\cD^\gamma(z))(\cD^\gamma)'(z) + z\phi_\cD''(\cD^\gamma(z))(\cD^\gamma)'(z)^2 + z\phi_\cD'(\cD^\gamma(z))(\cD^\gamma)''(z)
	\]
	and hence
	\[
	(\cD^\gamma)''(\rho_\cD) = 2 \phi_\cD'(\tau_\cD)(\cD^\gamma)'(\rho_\cD) + \rho_\cD\phi_\cD''(\tau_\cD)(\cD^\gamma)'(\tau_\cD)^2 + \rho_\cD\phi_\cD'(\tau_\cD)(\cD^\gamma)''(\rho_\cD).
	\]
	We know that $\rho_\cD \phi'_\cD(\tau_\cD) = \nu_\cD < 1$, $\phi'_\cD(\tau_\cD) < \infty$ and $(\cD^\gamma)'(\tau_\cD)< \infty$. Thus $(\cD^\gamma)''(\rho_\cD) = \infty$ if and only if $\phi''_\cD(\tau_\cD) = \infty$. It holds that
	\[
	\phi''_\cD(z) = \frac{s''(z)(1-s(z))^2 + 2 s'(z)(1-s(z)) }{(1 - s(z))^4}.
	\]
	Thus $\phi_\cD''(\tau_\cD) = \infty$ if and only if $s''(r)= \infty$. This concludes the proof.
\end{proof}

\subsection{Proof of Theorem~\ref{TEMAIN}}

\emph{Throughout the rest of Section~\ref{sec:palpha} we assume that our face-weights $\iota=(\iota_k)_k$ satisfy the requirements of Theorem~\ref{TEMAIN}.}

Our aim is to show that the random $\iota$-face-weighted outerplanar map $\mO_n^\omega$ converges towards a loop-tree after proper rescaling. For each outerplanar map $O$, we may construct the map $\bar{O}$ obtained by deleting all inner edges from $O$. That is, we remove all edges that do not lie on the frontier of the outer face.

The core of our arguments is the observation that $\mO_n^\omega$ is close to a rescaled version of $\bar{\mO}_n^\omega$.

\begin{lemma}
	\label{le:main1}
	It holds that 
	\[
		d_{\textsc{GH}}\left( (\bar{\mO}_n^\omega, (1-\nu_\cD)d_{\bar{\mO}_n^\omega}), (\mO_n^\omega, d_{\mO_n^\omega}) \right) = o_p((L_n n)^{1/\alpha}).
	\]
\end{lemma}

We are going to combine this with a scaling limit for the map $\bar{\mO}_n^\omega$:

\begin{lemma}
	\label{le:approx1}
	It holds that
	\[
	(\bar{\mO}_n^\omega, b_n^{-1} d_{\bar{\mO}_n^\omega}) \convdis (\mathscr{L}_\alpha, d_{\mathscr{L}_\alpha})
	\]
	for
	\[
		b_n = \left( \frac{nL_n\Gamma(-\alpha)}{1-s(r) }  \right)^{1/\alpha} \frac{1-r}{r}
	\]
\end{lemma}

Theorem~\ref{TEMAIN} then readily follows from these two Lemmas, recalling that $1 - \nu_\cD = r/(1-r)$ by Lemma~\ref{le:regime}.

\subsection{Proving Lemma~\ref{le:main1}}

First, we observe that for the weight-sequences under consideration we have explicit expressions for the coefficients of the generating series associated to dissections which by Lemma~\ref{le:outleaf} correspond to the branching weights $(p_n)_{n\geq 0}$ of $\tau_n^\cO$. The expression along with the relation to Galton--Watson processes is given in the following Lemma.
\begin{lemma}
	\label{le:asymptotics}
	The coefficients of $\phi_\cD(z)$ satisfy the asymptotic relation
	\begin {align*}
	[z^k]\phi_\cD(z) \sim \frac{L_k}{(1-s(r))^2} k^{-\alpha-1} r^{-k}
	\end {align*}
	as $k$ becomes large. 
	The tree  $\tau_n^\cO$ is distributed like a critical Galton--Watson tree conditioned on having $n$ leaves. Its offspring distribution $\xi$ satisfies $\Pr{\xi=0} = 1-r$ and
	\[
		\Pr{\xi=k}  \sim \frac{L_k(1-r)^{\alpha+1}}{(1-s(r))r^\alpha} k^{-\alpha -1}.
	\]
\end{lemma}
\begin{proof}
	We know by Lemma~\ref{le:regime} that $s(r)<1$. This allows us to apply \cite[Thm. 4.8, 4.30]{MR3097424} to deduce that the coefficients of $\phi_\cD(z)$ are asymptotically equal to the coefficients of $s(z)$ up to a constant factor, namely
	\begin{align*}
	[z^k]\phi_\cD(z) \sim \frac{1}{(1 - s(r))^2} L_k k^{-\alpha -1}r^{-k}.
	\end{align*}
	 Lemma~\ref{le:regime} also tells us that $\tau_\cD = r$. Hence 
	the tree $\cT^\cD_n$ from Lemma~\ref{le:dcoup} is distributed like a Galton--Watson tree $\cT^\cD$ conditioned on having  $n$ vertices, with offspring distribution $\xi(\cT^\cD)$ satisfying
	\begin{align}
		\label{eq:mref}
	\Pr{\xi(\cT^\cD) = k} = \frac{[z^k] \phi_\cD(rz)}{\phi_\cD(r)} \sim \frac{1}{1 - s(r)} L_k k^{-\alpha -1}.
	\end{align}
	By \cite[Eq. (14)]{MR3335012} this implies
	\begin{align*}
	\Pr{|\cT^\cD| = k} \sim \frac{1}{1 - s(r)} L_k \left( (1-\nu_\cD)k\right)^{-\alpha -1}
	\end{align*}
	with $1 - \nu_\cD = r/(1-r)$.
	On the other hand,
	$\cD^\gamma(\rho_\cD) = \tau_\cD$ implies
	\begin{align*}
	\Ex{z^{|\cT^\cD|}} = \frac{ \cD^\gamma(\rho_\cD z)}{ \cD^\gamma(\rho_\cD)} = \frac{\cD^\gamma(r(1-s(r)) z)}{r}.
	\end{align*}
	Recall that by Equation~\eqref{eq:sp} it holds that $\tau_\cO = \rho_\cD$. Hence the tree $\tau_n^\cO$ from Lemma~\ref{le:outleaf} is distributed like a Galton--Watson tree  conditioned on having $n$ leaves, with offspring distribution $\xi$ satisfying
	\begin{align*}
	\Pr{\xi = 0} = 1 - \cD^\gamma(\rho_\cD) = 1- r
	\end{align*}
	and
	\begin{align*}
	\Pr{\xi = k} &= \rho_\cD^{k-1} [z^{k-1}]\cD^\gamma(z) = r \Pr{|\cT^\cD| = k-1} \notag \\ &\sim \frac{r}{1 - s(r)} L_k \left( (1-\nu_\cD)k\right)^{-\beta} \notag \\
	&= \frac{L_k(1-r)^{\alpha+1}}{(1-s(r))r^\alpha} k^{-\alpha -1}.
	\end{align*}
	The offspring distribution is critical, since
	\[
	1 = \nu_\cO = \frac{\rho_\cD (\cD^\gamma)'(\rho_\cD)}{1 - \cD^\gamma(\rho_\cD)}
	\]
	implies that
	\[
		\Ex{\xi} = \cD^\gamma(\rho_\cD) + \rho_\cD (\cD^\gamma)'(\rho_\cD) = 1 + (1 - \cD^\gamma(\rho_\cD)) = 1.
	\]
\end{proof}

Any dissection $D$ shares the same set of vertices as the circle $\bar{D}$, and hence there is a canonical correspondence between the two. We define the penalty function $f(D)$ as the distortion of this correspondence if we rescale $\bar{D}$ by the factor $1- \nu_\cD$. That is,
\[
f(D) = \max_{v,v' \in D} \left| (1-\nu_\cD) d_{\bar{D}}(v,v') - d_D(v,v') \right|.
\]
Note that there is a trivial upper bound
\begin{align}
	\label{eq:trivial}
	f(D) \le 2 |D|.
\end{align}

We observe that $\iota$-face-weighted dissections converge in the Gromov--Hausdorff sense towards a deterministic circle. The idea is that $\mD_n^\gamma$ has a giant face of size roughly $(1-\nu_\cD)n$ and the sizes of the dissections attached to its boundary behave asymptotically in an i.i.d. manner.

\begin{figure}[t]
	\centering
	\begin{minipage}{1.0\textwidth}
		\centering
		\includegraphics[width=0.32\textwidth]{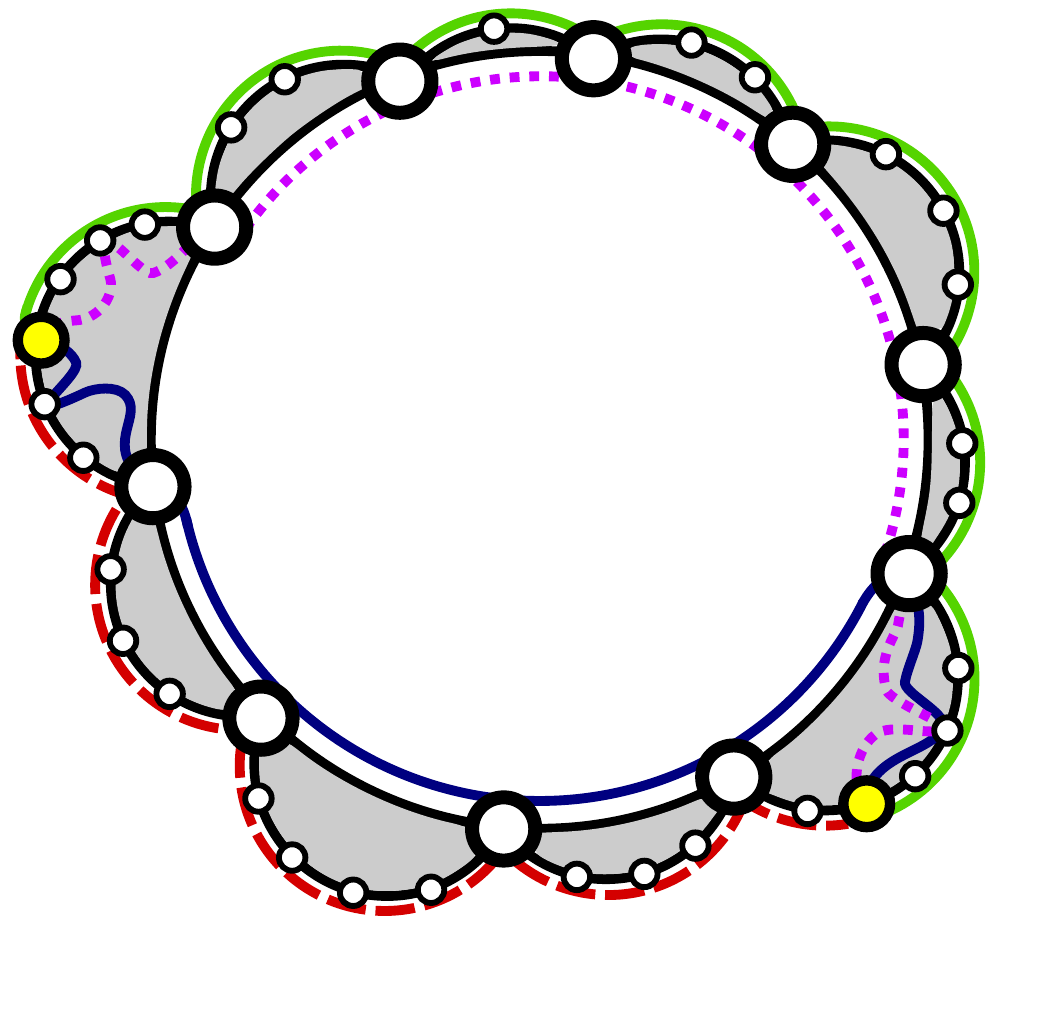}
		\caption{A dissection with a giant face that is depicted as a white disc. The shaded regions represent smaller dissections attached to its frontier. The possible geodesics between the two marked vertices are coloured in solid blue and dotted magenta. The candidates for the corresponding geodesics on the boundary are coloured in dashed red and solid green.}
		\label{fi:diss}
	\end{minipage}
\end{figure}  

\begin{lemma}
	\label{le:diss}
	There is a slowly varying sequence $\tilde{L}_n$ such that $f(\mD_n^\gamma) = O_p( \tilde{L}_n n^{1/\alpha})$. 
\end{lemma}
\begin{proof}
	Recall that $\mD_n^\gamma$ corresponds via a combinatorial bijection to the enriched tree $(\cT_n^\cD, \beta_n^\cD)$, where we assign to each vertex $v \in \cT_n^\cD$ a dissection $\beta_n^\cD(v)$ whose chords are required to be incident to the destination of the root edge. Let $u^* \in \cT_n^\cD$ denote the lexicographically first vertex with maximal outdegree. Let $A_n$ denote the tree obtained by cutting away all descendants of the vertex $u^*$ and let $(T_n^i)_{1 \le i \le d^+_{\cT_n^\cD}(u^*)}$ denote the ordered family of fringe subtrees dangling from $u^*$. 
	
	By Equation~\eqref{eq:mref} we know that the tree $\cT_n^\cD$ is distributed like a Galton--Watson tree $\cT^\cD$ conditioned on having size $n$, with offspring distribution $\xi(\cT^\cD)$ following a power law up to a slowly varying factor and satisfying $\Ex{\xi(\cT^\cD)} = \nu_\cD < 1$. By~\cite[Thm. 20.1]{MR2908619} it follows that the pruned tree $A_n$ converges in the local weak sense towards an almost surely finite tree, and in particular the size of $A_n$ is stochastically bounded. As for the fringe subtrees, Kortchemski~\cite{{MR3335012}} observed in the more general context of simply generated trees in the condensation regime  that the vector $(T_n^i)_{1 \le i \le d^+_{\cT_n^\cD}(u^*)}$ may be approximated by a vector of independent copies of $\cT^\cD$. 
	His results~\cite[Thm. 1, Cor. 1]{MR3335012} imply that there is a slowly varying sequence $(\tilde{L}_n)_{n \ge 1}$ such that
	\begin{align}
	\label{eq:x1}
	d^+_{\cT_n^\cD}(u^*) = (1 - \nu_\cD)n + O_p(\tilde{L}_n n^{1/\alpha}) 
	\end{align}
	and
	\begin{align}
	\label{eq:x2}
	\max\{ |T_n^i| \mid 1 \le i \le d^+_{\cT_n^\cD}(u^*) \} = O_p( \tilde{L}_n n^{1/\alpha}).
	\end{align}
	Summing up, we have obtained that the dissection $\mD_n^\gamma$ consists of a giant chord restricted dissection with small dissections of maximal size $O_p(\tilde{L}_n n^{1/\alpha})$ attached to its boundary. 
	Moreover, by a general Gibbs partition result~\cite[Thm. 3.1]{Mreplaceme} it follows that the largest face in the chord restricted dissection $\beta_n^\cD(u^*)$ has size $d^+_{\cT_n^\cD}(u^*) + O_p(1)$. In other words, $\mD_n^\gamma$ has a giant face $F$ of size $(1 - \nu_\cD + o_p(1))n$, and the dissections attached to the boundary of this face have maximum size $O_p( \tilde{L}_n n^{1/\alpha})$. Compare with Figure~\ref{fi:diss}. Any point of the dissection that does not already lie on the boundary of $F$ is contained in a unique dissection attached to the boundary of $F$. Thus, for any vertex $v$ of $\mD_n^\gamma$ there are at most two vertices on the boundary of $F$ whose distance is minimal from $v$. We pick any of the at most two and call it $v_F$. By Equation~\eqref{eq:x2} we obtain
	\[
		\sup_{v \in \mD_n^\gamma} d_{\mD_n^\gamma}(v, v_F) \le \max\{ |T_n^i| \mid 1 \le i \le d^+_{\cT_n^\cD}(u^*) \} = O_p( \tilde{L}_n n^{1/\alpha}).
	\]
	In particular, if $d_F$ denotes the metric on the circle $F$, we obtain that
	\begin{align}
		\label{eq:x3}
		\sup_{u,v \in \mD_n^\gamma} |d_{\mD_n^\gamma}(u,v) - d_{F}(u_F,v_F)| = O_p( \tilde{L}_n n^{1/\alpha}).
	\end{align}

	In order to bound $f(\mD_n^\gamma)$, we have to compare the distance $d_{\mD_n^\gamma}(u,v)$ with the scaled distance $(1- \nu_\cD)d_{\bar{\mD}_n^\gamma}(u,v)$ on the boundary. The strategy is that if we need to pass through $k$ edges in order to travel from $u_F$ to $v_F$ on the boundary of $F$ in clockwise order, then we have to pass through roughly $k/(1- \nu_\cD)$ edges in order to travel from $u$ to $v$ in clockwise order on the boundary  $\bar{\mD}_n^\gamma$. And the idea behind this thought is that the  dissections encountered along the way (except the one corresponding to $A_n$, and the one containing the small faces of $\beta_n^\cD(u^*)$)  behave like i.i.d. copies of a dissection of a polygon with circumference $|\cT^\cD|+1$, and it holds that $\Ex{|\cT^\cD|} = 1/(1 - \nu_\cD)$ as $\Ex{\xi^\cD} = \nu_\cD< 1$.

	Let us make this explicit. It was shown in~\cite[Thm. 3]{MR3335012}  that the process $Z_k := |T_n^1| + \ldots + |T_n^k|$ admits a scaling limit
	\[
		\left( \frac{Z_{\lfloor t d^+_{\cT_n^\cD}(u^*) \rfloor} - t d^+_{\cT_n^\cD}(u^*)/(1 - \nu_\cD) }{ \tilde{L}_n n^{1/\alpha} } \right) \convdis (Y_t)_{0 \le t \le 1}
	\]
	in the Skorokhod space $\mathbb{D}([0,1], \ndR)$. Hence
	\begin{align}
		\label{eq:yo}
		\max_{1 \le k \le d^+_{\cT_n^\cD}(u^*)} | Z_k - k /(1 - \nu_\cD)| = O_p( \tilde{L}_n n^{1/\alpha}).
	\end{align}
	We may write
	\[
		d_F(u_F, v_F) = \min\left(d^\ell_F(u_F, v_F), d^r_F(u_F, v_F) \right).
	\]
	with $d^\ell(u_F, v_F)$ and $d^r(u_F, v_F)$ denoting the number of edges required to traverse from $u_F$  to $v_F$ on the boundary of $F$ in a clockwise and counter-clockwise manner, respectively. Likewise, we may write
	\[
		d_{\bar{\mD}_n^\gamma}(u,v) = \min\left(d_{\bar{\mD}_n^\gamma}^\ell(u, v), d^r_{\bar{\mD}_n^\gamma}(u, v) \right)
	\]
	with $d^\ell_{\bar{\mD}_n^\gamma}(u,v)$ and $d^r_{\bar{\mD}_n^\gamma}(u,v)$ denoting the number of edges required to traverse from $u$ to $v$ on the circle $\bar{\mD}_n^\gamma$ in a clockwise or counter-clockwise manner. It follows from Equation~\eqref{eq:yo} that
	\[
		\sup_{u,v \in \mD_n^\gamma} | d_{\bar{\mD}_n^\gamma}^\ell(u, v) - d^\ell_F(u_F, v_F)/(1- \nu_\cD)| = O_p( \tilde{L}_n n^{1/\alpha})
	\]
	and likewise for $d_{\bar{\mD}_n^\gamma}^r$ and $d_F^r$. Consequently, it holds that
	\[
	\sup_{u,v \in \mD_n^\gamma} | d_{\bar{\mD}_n^\gamma}(u, v) - d_F(u_F, v_F)/(1- \nu_\cD)| = O_p( \tilde{L}_n n^{1/\alpha}).
	\]
	By Equation~\eqref{eq:x3} this means that
	\[
		f(\mD_n^\gamma) = \sup_{u,v \in \mD_n^\gamma} |(1- \nu_\cD) d_{\bar{\mD}_n^\gamma}(u, v) -  d_{{\mD}_n^\gamma}(u, v)| = O_p( \tilde{L}_n n^{1/\alpha}).
	\]
\end{proof}

\begin{remark}
\label{eq:rem}
Let $\mu_n$ denote the uniform distribution on the vertices of $\mD_n^\gamma$, and $\mu$ the uniform distribution on the circle $C^1$. Lemma~\ref{le:diss} implies that
\[
\left(\mD_n^\gamma, \frac{1}{ n(1-\nu_\cD)} d_{\mD_n^\gamma}, \mu_n \right) \convdis (C^1, d_{C^1}, \mu)
\]
in the Gromov--Hausdorff--Prokhorov sense.
\end{remark}

We denote the height of a plane tree $T$ by  $\He(T)$. Recall that the Lukasiewicz path $W(T) = (W_k(T))_{1 \le k \le |T|}$ of  $T$ is defined by ordering its vertices in depth-first-search order $v_1, \ldots, v_{|T|}$ and setting
\[
	W_k(T) = \sum_{i=1}^{k-1}(d_T^+(v_i) -1), \quad 1 \le k \le |T|.
\]
The depth-first-search always tries to proceed to the left-most offspring of the current vertex. We may also consider the reverse depth-first-search list $\hat{v}_1, \ldots, \hat{v}_{|T|}$ of vertices of $T$, where in each step we try to proceed with the right-most offspring instead. We define the mirrored Lukasiewicz path $\hat{W}(T) = (\hat{W}_k(T))_{1 \le k \le |T|}$ by
\[
\hat{W}_k(T) = \sum_{i=1}^{k-1} \left(d^+_{T}(\hat{v}_i) - 1\right)
\]
We call a random plane tree $\mT$ mirror invariant, if $W(\mT) \eqdist \hat{W}(\mT)$. This is equivalent to stating that the \emph{distribution} of $\mT$ does not change if we reverse the ordering in each offspring set.
%\end{definition}

We let  $X^{\mathrm{exc},(\alpha)} = ( X^{\mathrm{exc},(\alpha)}_t)_{0 \le t \le 1}$ denote the normalized excursion of the $\alpha$-stable spectrally positive L\'evy process. We refer to Bertoin's book~\cite{MR1406564} for background information on L\'evy processes, and to \cite{MR3286462} for details on this particular process. This process lives on the Skorokhod space $D([0,1], \ndR)$ of real valued functions that are  c\`adl\`ag, that is, they are continuous from the right and have left-side limits. See \cite[Ch. VI]{MR959133}  for details on this classical space of functions.

Lemma~\ref{le:diss} ensures that "large" dissections in $\mO_n^\omega$ asymptotically look like circles. The following result will aid us in showing that "small" dissections do not contribute to the asymptotic geometric shape. Its proof is based on the proof of \cite[Thm. 4.1]{MR3286462}.

\begin{lemma}
	\label{le:core}
	 Let $(\mT_n)_{n \ge 1}$ denote a sequence of  mirror invariant random finite plane trees. Suppose that there exists a sequence $B_n$ of positive real numbers such that the Lukasiewicz path $(W_k^n)_{1 \le k \le |\mT_n|} := W_k(\mT_n)$ corresponding to $\mT_n$ satisfies
	\begin{align}
		\label{eq:conv}
		\left( \frac{1}{B_n} W_{ \lfloor t |\mT_n|\rfloor }^n \right)_{0 \le t \le 1} \convdis  X^{\mathrm{exc},(\alpha)}
	\end{align}
	in the Skorokhod space $D([0,1], \ndR)$. Suppose that additionally 
	\begin{align}
	\He(\mT_n) = o_p(B_n).
	\end{align}
	Then for any $\epsilon>0$ there is a $\delta >0$ such that for all large enough $n$
	\begin{align}
		\label{eq:AA}
		\Prb{ \max_{v \in \mT_n} \sum_{u \text{ ancestor of } v} d^+_{\mT_n}(u) \one_{d^+_{\mT_n}(v) \le \delta B_n}  \ge \epsilon B_n} < \epsilon.
	\end{align}
\end{lemma}
\begin{proof}
	Let $v_1, \ldots, v_{|\mT_n|}$ denote the depth-first-search ordered list of vertices of the tree $\mT_n$. For any vertex $v \in \mT_n$ we may consider the indices $1 = i_1 < i_2  < \ldots < i_k$ such that $v_{i_1}, \ldots, v_{i_k}$ is the path from the root $v_1$ to the vertex $v = v_{i_k}$. Of course, the indices depend on $v$, but for the sake of readability we are not going to denote this explicitly. (The inclined reader may imagine an invisible superscript $v$  on each index and on $k$, that is, $i_\ell^v$ instead of $i_\ell$, and $k^v$ instead of $k$.)
	
	Let us call any sibling of a vertex in a plane tree that lies to its right a "right-sibling", and likewise any sibling that lies to its left a "left-sibling". Then $W_{i_k}^n$ counts the number of right-siblings of the ancestors of $v$. Moreover, as $W_1^n = 0$, we may write
	\[
		W_{i_k}^n = \sum_{\ell =2}^k (W_{i_\ell}^n - W_{i_{\ell -1}}^n)
	\]
	Here $W_{i_\ell}^n - W_{i_{\ell -1}}^n$ counts the number of right-siblings of $v_{i_\ell}$. Similarly, if $\hat{v}_1, \ldots, \hat{v}_{|\mT_n|}$ is the reverse depth-first-search ordered list of vertices of $\mT_n$, we may consider the indices $1=j_1 < j_2 < \ldots < j_h$ such  that $\hat{v}_{j_1}, \ldots, \hat{v}_{j_h}$ is the path from the root $\hat{v}_1$ to the vertex $v = \hat{v}_{j_h}$. We may write
\begin{align*}
\sum_{u \text{ ancestor of } v} d_{\mT_n}^+(u) = W_{i_k}^n + \hat{W}_{j_h}^n + \he_{ \mT_n}(v),
\end{align*}
	with $\he_{\mT_n}(v)$ denoting the height of the vertex $v$ in the tree $\mT_n$.  If the degree of an ancestor $u$ of $v$ is at most $\delta B_n$, then $u$ has at most $\delta B_n$ left-siblings and at most $\delta B_n$ right-siblings. This allows us to write for any $\delta>0$
\begin{align*}
\sum_{u \text{ ancestor of } v} d_{\mT_n}^+(u)\one_{d^+_{\mT_n}(v) \le \delta B_n} \le \He(\mT_n) &+  \sum_{\ell =2}^k (W_{i_\ell}^n - W_{i_{\ell -1}}^n)\one_{(W_{i_\ell}^n - W_{i_{\ell -1}}^n) \le \delta B_n  } \\
&+ \sum_{\ell =2}^h (\hat{W}_{j_\ell}^n - \hat{W}_{j_{\ell -1}}^n)\one_{(\hat{W}_{j_\ell}^n - \hat{W}_{j_{\ell -1}}^n) \le \delta B_n  }.
\end{align*}
For any $\epsilon>0$ it holds that if the left-side of this inequality is at least $\epsilon B_n$, then at least one of the three summands on the right-hand side is at least $\epsilon B_n / 3$. Hence
\begin{multline}
\label{eq:tt}
\Prb{ \max_{v \in \mT_n} \sum_{\ell = 1}^{k-1} d^+_{\mT_n}(v_{i_\ell}) \one_{d^+_{\mT_n}(v_{i_\ell}) \le \delta B_n}  \ge \epsilon B_n} \le \Pr{  \He(\mT_n) \ge \epsilon B_n / 3} \\
+ 2 \Prb{  \max_{v \in \mT_n}  \sum_{\ell =2}^k (W_{i_\ell}^n - W_{i_{\ell -1}}^n)\one_{(W_{i_\ell}^n - W_{i_{\ell -1}}^n) \le \delta B_n  } \ge \epsilon B_n / 3 }.
\end{multline}

We also assumed that $\He(\mT_n) = o_p(B_n)$, so the first summand on the right hand side tends to zero as $n$ becomes large. Hence in order to verify Inequality~\eqref{eq:AA}, we need to show that we may choose $\delta$ sufficiently small (depending on $\epsilon$) such that the second summand in \eqref{eq:tt} is smaller than $\epsilon/2$ for sufficiently large $n$.

We are going to prove this by contradiction. Suppose that there is an $\epsilon>0$ such that for each $\delta>0$ it happens for infinitely many $n$ that there is a vertex $v \in \mT_n$ with
\begin{align}
\label{eq:contradiction}
\Prb{ \sum_{\ell =2}^k (W_{i_\ell}^n - W_{i_{\ell -1}}^n)\one_{(W_{i_\ell}^n - W_{i_{\ell -1}}^n) \le \delta B_n / 3 } \ge \epsilon B_n / 3 } > \epsilon/2.
\end{align}
(Recall that the indices $i_1, \ldots, i_k$ and $k$ depend on $v$, but we suppress the super-scripted $v$ in order to improve readability.) It follows that there is a sequence of positive numbers $\delta_n \to 0$ and a subsequence of the natural numbers such that Inequality~\eqref{eq:contradiction} holds for $\delta = \delta_n$ as $n$ tends to infinity along this subsequence. The random variable $k / |\mT_n|$ lies in the unit interval. The space of Borel-probability measures on the compact unit interval is compact in the topology of weak convergence. It follows  that, by passing to another subsequence, we may without loss of generality assume  that there is a random number $t_0 \in [0,1]$ such that $k / |\mT_n| \convdis t_0$.

For all $f \in D([0,1], \ndR)$ and $a,b \in [0,1]$  with $a \le b$ let us write $a \preccurlyeq_f b$ if  \[x_a^b(f) := \inf_{x \in [a,b]} f(x) -  f(a-) \ge 0.\] 
It was shown in \cite[Cor. 3.4]{MR3286462} that  almost surely  for all $b \in [0,1]$	
\[
X^{\mathrm{exc},(\alpha)}_b = \sum_{0 \preccurlyeq_{X^{\mathrm{exc},(\alpha)}} a \preccurlyeq_{X^{\mathrm{exc},(\alpha)}} b} x_a^b(X^{\mathrm{exc},(\alpha)}).
\]
This sum is finite for all $b$ since $X^{\mathrm{exc},(\alpha)}$ is a bounded function, just like any other c\`adl\`ag function on a compact interval. Hence
\begin{align}
\label{eq:t5}
\Prb{ \sum_{0 \preccurlyeq s \preccurlyeq {t_0}} x_s^{t_0}(X^{\mathrm{exc},(\alpha)}) \one_{x_s^{t_0}(X^{\mathrm{exc},(\alpha)}) \le \delta_n}  \ge \epsilon/3} =o(1).
\end{align}
On the other hand, setting $w_n := \left( \frac{1}{B_n} W_{ \lfloor t |\mT_n|\rfloor }^n \right)_{0 \le t \le 1}$ it holds that
\[
B_n^{-1}\sum_{\ell =2}^k (W_{i_\ell}^n - W_{i_{\ell -1}}^n)\one_{(W_{i_\ell}^n - W_{i_{\ell -1}}^n) \le \delta_n B_n } = 	\sum_{0 \preccurlyeq_{w_n} a \preccurlyeq_{w_n} k/|\mT_n|} x_a^{k/|\mT_n|}(w_n) \one_{x_a^{k/|\mT_n|}(w_n) \le \delta_n }.
\]
It follows from this and \eqref{eq:t5}, from the limit~\eqref{eq:conv}, and \cite[second display after Equation~(4.8)]{MR3286462},  that the probability on the left-hand side of Inequality~\eqref{eq:contradiction} tends to zero as $n$ becomes large along the subsequence we fixed. We have thus arrived at the desired contradiction, completing the proof.
\end{proof}

	The limit~\eqref{eq:conv} and properties of the Skorokhod topology entail that the jumps  of the rescaled Lukasiewicz path of $\mT_n$ converge towards the jumps of  $X^{\mathrm{exc},(\alpha)}$. Like any c\`adl\`ag function on a compact interval the excursion $X^{\mathrm{exc},(\alpha)}$ has only finitely many jumps of height at least $\epsilon$. This shows that for any $\epsilon>0$ it holds that
\begin{align}
\label{eq:BB}
|\{v \in \mT_n \mid d^+_{\mT_n}(v) \ge \epsilon B_n\}| = O_p(1).
\end{align}
The distribution of the size of the largest jump (that is, the limit distribution of the rescaled maximum degree $(\Delta(\mT_n)-1) / B_n$) is given in \cite[Formula (19.97)]{MR2908619}. For our purposes, it will be enough to know that the limit \eqref{eq:conv} implies
\begin{align}
\label{eq:CC}
\Delta(\mT_n) = O_p(B_n).
\end{align}
We now have all the ingredients for proving Lemma~\ref{le:main1}. 

\begin{proof}[Proof of Lemma~\ref{le:main1}]
	Setting
	\begin{align*}
	\tilde{B}_n = |\Gamma(1 - \alpha)|^{1/\alpha} \inf\{x \ge 0 \mid \Pr{\xi > x } \le 1/n\}
	\end{align*}
	it follows from Lemma~\ref{le:asymptotics} and  \cite[Thm. 6.1]{MR2946438} that the Lukasiewicz path $(\cW_k(\tau_n^\cO))_{0 \le k \le |\tau_n^\cO|}$ of the tree $\tau_n^\cO$ satisfies
	\begin{align*}
	\left( \frac{1}{\tilde{B}_{|\tau_n^\cO|}} \cW_{\lfloor t |\tau_n^\cO| \rfloor} (\tau_n^\cO )\right)_{0\le t \le 1} \convdis X^{\mathrm{exc}, \alpha}.
	\end{align*}
	By \cite[Rem. 5.10]{MR2946438} the scaling factor $\frac{1}{\tilde{B}_{|\tau_n^\cO|}}$ may be replaced by $\frac{1}{B_n}$ with
	\begin{align*}
	B_n = \frac{\tilde{B}_n}{\Pr{\xi=0}^{1/\alpha}} = \left( \frac{|\Gamma(1 - \alpha)|}{1-r}\right)^{1/\alpha} \inf\{x \ge 0 \mid \Pr{\xi > x } \le 1/n\} .
	\end{align*}
	It follows from Karamata's theorem that 
	\[
	\Pr{\xi>k} \sim \frac{L_k(1-r)^{\alpha+1}}{(1-s(r))r^\alpha \alpha} k^{-\alpha}
	\]
	as $k$ becomes large. Hence
	\[
	\inf\{x \ge 0 \mid \Pr{\xi > x } \le 1/n\} \sim \left( \frac{ L_n(1-r)^{\alpha+1}}{(1-s(r))r^\alpha \alpha} n \right)^{1/\alpha}.
	\]
	Consequently, using $\Gamma(1-\alpha) = -\Gamma(-\alpha) \alpha$,
	\[
	B_n \sim \left( \frac{nL_n\Gamma(-\alpha)}{1-s(r) }  \right)^{1/\alpha} \frac{1-r}{r} = b_n.
	\]
	Summing up, we have that 
	\begin{align}
	\label{eq:first}
	\left( \frac{1}{b_n} \cW_{\lfloor t |\tau_n^\cO| \rfloor} (\tau_n^\cO )\right)_{0\le t \le 1} \convdis X^{\mathrm{exc}, \alpha}.
	\end{align}
	Likewise, \cite[Thm. 5.9, Rem. 5.10]{MR2946438} and $1 < \alpha < 2$ imply that \begin{align}
	\label{eq:second}
	\He(\tau_n^\cO) = o_p(b_n).
	\end{align} Hence in the following we may apply Lemma~\ref{le:core} to the random tree $\tau_n^\cO$.

	For any two vertices $u, v \in \tau_n^\cO$ let  $P(u,v)$ denote the unique path from $u$ to $v$ in the tree $\tau_n^\cO$. 
	The distortion of the canonical correspondence between $\mO_n^\omega$ and the rescaled map $(1-\nu_\cD)\bar{\mO}_n^\omega$ is bounded  by the maximum sum of penalties of the dissections along the  paths in the tree $\tau_n^\cO$.
	Hence 
	\begin{align}
		\label{eq:q1}
		d_{\textsc{GH}}\left( (\bar{\mO}_n^\omega, (1-\nu_\cD)d_{\bar{\mO}_n^\omega}), (\mO_n^\omega, d_{\mO_n^\omega}) \right) \le \max_{u, v \in \tau_n^\cO} \sum_{x \in P(u,v)} f(\delta_n^\cO(x)).
	\end{align}
	The path $P(u,v)$ passes through the youngest common ancestor~$a$ of the vertices $u$ and $v$. If for any vertex $x \in \tau_n^\cO$ we write $S_x$ for the sum of penalties along the path from the root to $x$, then
	\[
		\sum_{x \in P(u,v)} f(\delta_n^\cO(x)) = S_u + S_v - 2 S_a + f(\delta_n^\cO(a)).
	\]
	It follows from Inequality~\eqref{eq:q1} that
	\begin{align}
	\label{eq:t1}
	d_{\textsc{GH}}\left( (\bar{\mO}_n^\omega, (1-\nu_\cD)d_{\bar{\mO}_n^\omega}), (\mO_n^\omega, d_{\mO_n^\omega}) \right) \le 3 \max_{v \in \tau_n^\cO} S_v .
	\end{align}
	Let $\delta$ be an arbitrary positive number. For any vertex $v$ we may write 
	\[
		S_v = \sum_{u \text{ ancestor of } v} f(\delta_n^\cO(u)) = S_v^{< \delta} + S_v^{ \ge \delta}
	\]
	with $S_v^{< \delta}$ denoting the sum $S_v$ restricted to all $u$ with degree smaller than $\delta b_n$, and analogously $S_v^{\ge \delta}$ denoting the sum restricted to all $u$ with degree at least $\delta b_n$. Let $\epsilon > 0$ be given. By Lemma~\ref{le:core} and Inequality~\eqref{eq:trivial} we may choose $\delta$ small enough (depending on~$\epsilon$) such that for sufficiently large $n$
	\begin{align}
		\label{eq:t2}
		\Pr{\max_{v \in \tau_n^\cO} S_v^{< \delta} \ge \epsilon b_n} < \epsilon.
	\end{align}
	By Equation~\eqref{eq:BB} the number of vertices in the tree $\tau_n^\cO$ with outdegree at least $\delta b_n$ is stochastically bounded. Hence there is a positive large number $M$ such that for large enough~$n$
	\[
		\Pr{ | \{ v \in \tau_n^\cO \mid d^+_{\tau_n^\cO}(v) \ge \delta b_n\} | > M } < \epsilon.
	\]
	By Equation~\eqref{eq:CC} we may assume that $M$ is also large enough such that for all $n$
	\[
		\Pr{\Delta(\tau_n^\cO) \ge M b_n} < \epsilon.
	\]
	It follows that for sufficiently large $n$ we may bound the probability $\Pr{\max_{v \in \tau_n^\cO} S_v^{\ge \delta} \ge \epsilon b_n}$ by
	\begin{align*}
		  2\epsilon + \Pr{\max_{v \in \tau_n^\cO} S_v^{\ge \delta} \ge \epsilon b_n, | \{ v \in \tau_n^\cO \mid d^+_{\tau_n^\cO}(v) \ge \delta b_n\} | < M, \Delta(\tau_n^\cO) \le M b_n}.
	\end{align*}
	By Lemma~\ref{le:diss} we know that $f(\mD_k^\gamma) = o_p(k)$ as $k \to \infty$. Hence if there are at most $M$ large  dissections in $\mO_n^\omega$ (with "large" meaning having at least $\delta b_n$ vertices), each having size less than $M b_n$, then the sum of their penalties is $o_p(b_n)$. Hence
	\begin{align}
		\label{eq:t3}
		\Pr{\max_{v \in \tau_n^\cO} S_v^{\ge \delta} \ge \epsilon b_n} \le 3 \epsilon
	\end{align}
	for large enough $n$.	
	As $\epsilon>0$ was arbitrary, Inequalities~\eqref{eq:t1}, \eqref{eq:t2}, and \eqref{eq:t3} imply that
	\[
d_{\textsc{GH}}\left( (\bar{\mO}_n^\omega, (1-\nu_\cD)d_{\bar{\mO}_n^\omega}), (\mO_n^\omega, d_{\mO_n^\omega}) \right) = o_p(b_n) = o_p((L_n n)^{1/\alpha}).
\]
\end{proof}

\subsection{Proving Lemma~\ref{le:approx1}}

The idea for proving Lemma~\ref{le:approx1} is to combine the coupling of Lemma~\ref{le:outleaf} with scaling limit results for random looptrees by Curien and Kortchemski~\cite{MR3286462}. The latter may be applied due to the asymptotic expansions in Lemma~\ref{le:asymptotics}.

\begin{figure}[t]
	\centering
	\begin{minipage}{1.0\textwidth}
		\centering
		\includegraphics[width=0.75\textwidth]{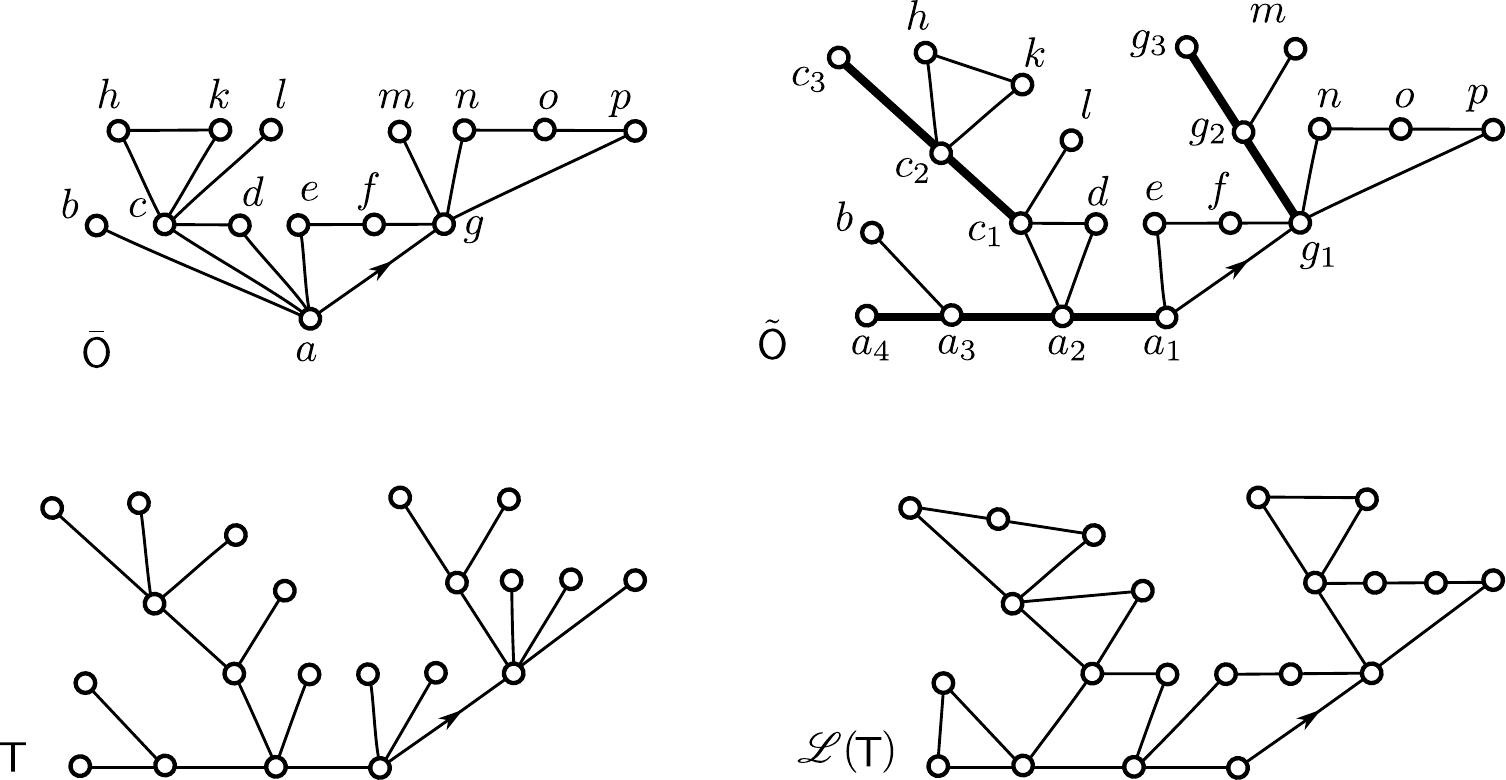}
		\caption{On the top left is an example of a chord-less outerplanar map $\bar{\mO}$. The outerplanar map $\tilde{\mO}$ on the top right is obtained from $\bar{\mO}$ by expanding each of the cutvertices in $\bar{\mO}$ into a line graph which is depicted with a thick line in $\tilde{\mO}$. A correspondence is built by letting each cutvertex correspond to the vertices of the thick line, e.g.~the vertex $a$ corresponds to the vertices $a_1,a_2,a_3,a_4$. The associated tree $\mT$ and discrete looptree $\mathscr{L}(\mT)$ are depicted on the bottom line.}
		\label{fi:gh}
	\end{minipage}
\end{figure}

\begin{proof}[Proof of Lemma~\ref{le:approx1}]
	By Lemma~\ref{le:outleaf}, we may sample the outerplanar map $\mO_n^\omega$ such that it corresponds to a simply generated tree with leaves as atoms $\tau_n^\cO$ that is decorated with random dissections of polygons. Let $\mathscr{L}(\tau_n^\cO)$ denote the discrete loop-tree corresponding to $\tau_n^\cO$.
	
	We claim that
	\begin {align} \label{eq:ghinequality}
	d_{\textsc{GH}} \left( (\bar{\mO}_n^\omega, d_{\bar{\mO}_n^\omega}), ( \mathscr{L}(\tau_n^\cO), d_{\mathscr{L}(\tau_n^\cO)}) \right) \le 2 \He(\tau_n^\cO)+1.
	\end{align}	
	To see this, we define an intermediate object which will clarify how to choose a good correspondence between the vertex sets of $\bar{\mO}_n$ and  $\mathscr{L}(\tau_n^\cO)$. 
	
	Let $\bar{\mathsf{O}}$ be an outerplanar map in which each dissection is simply a polygon without any chords. Let $\mathsf{T}$ be the corresponding tree with leaves as atoms as defined in Subsection \ref{ss:optrees}, and let $\mathscr{L}(\mathsf{T)}$ be the corresponding discrete looptree. We define an intermediate object $\tilde{\mathsf{O}}$ by expanding each cutvertex of $\bar{\mathsf{O}}$ into a line graph with a length which equals the number of blocks attached to the cutvertex. The corresponding blocks are then attached, one by one, to this line in the same order as they appear around the cutvertex. Each line will have one remaining endpoint of degree one which no block is attached to. See top of Fig.~\ref{fi:gh}
	for an illustration. 
	
	We define a correspondence between $\bar{\mathsf{O}}$ and $\tilde{\mathsf{O}}$ by letting each cutvertex in $\bar{\mathsf{O}}$ correspond to vertices on the associated line graph in $\tilde{\mathsf{O}}$ and other vertices have an obvious correspondence. The tree $T$ is obtained from $\tilde{\mO}$ by keeping the same vertex set and modifying the edges in a simple way as is evident from Fig.~\ref{fi:gh}. The discrete looptree $\mathscr{L}(\mT)$, which also shares the same vertex set as $\mT$, will then have a canonical correspondence with $\tilde{\mO}$.
	  It is now straightforward to see that the Gromov--Hausdorff distance between $\bar{\mO}$ and $\tilde{\mO}$ is at most $2\mH(\mT)$ and that the Gromov--Hausdorff distance between $\tilde{\mO}$ and $\mathscr{L}(\mT)$ is at most 1. Then \eqref{eq:ghinequality} follows from the triangle inequality.
	
	 By Equation~\eqref{eq:first} and the scaling limit~\cite[Thm. 4.1]{MR3286462}, it follows that
	\[
	(\mathscr{L}(\tau_n^\cO), b_n^{-1} d_{\mathscr{L}(\tau_n^\cO)}) \convdis (\mathscr{L}_\alpha, d_{\mathscr{L}_\alpha}).
	\]
	Consequently, by the inequality \eqref{eq:ghinequality} and by Equation~\eqref{eq:second} which tells us that $\He(\tau_n^\cO) = o_p(b_n)$ it follows that
	\[
	(\bar{\mO}_n^\omega, b_n^{-1} d_{\bar{\mO}_n^\omega}) \convdis (\mathscr{L}_\alpha, d_{\mathscr{L}_\alpha}).
	\]
\end{proof}

\section{Proofs in the circle regime}

\begin{proof}[Proof of Theorem~\ref{te:circ}]
By arguments identical to those in the proof of Lemma~\ref{le:asymptotics}, we see that the tree $\cT^\cD_n$ is distributed like a Galton--Watson tree conditioned on having $n$ vertices, with offspring distribution $\xi(\cT^\cD)$ satisfying
\begin{align*}
	\Pr{\xi(\cT^\cD) = k} \sim \frac{1}{1 - s(r)} L_k k^{-\alpha -1},
\end{align*}
and that the tree $\tau_n^\cO$ is distributed like a Galton--Watson tree conditioned on having $n$ leaves, with branching distribution $\xi$ satisfying
	\begin{align*}
\Pr{\xi = k} &\sim \frac{L_k(1-r)^{\alpha+1}}{(1-s(r))r^\alpha} k^{-\alpha -1}.
\end{align*}	
As
\[
\Pr{\xi = k} = \rho_\cD^{k-1} [z^{k-1}]\cD^\gamma(z)
\]
and $\Pr{\xi = k} \sim \Pr{\xi = k-1}$,
this implies
that
\[
[z^k]\cD^\gamma(z) \sim \rho_\cD^{-k} \Pr{\xi =k}.
\]
Note that $\nu_\cD <1$ and Lemma~\ref{le:reg2} imply $\cD^\gamma(\rho_\cD) = \tau_\cD = r < 1$. This allows us to apply \cite[Thm. 4.8, 4.30]{MR3097424} to deduce that
\begin{align*}
[z^k]\phi_\cO(z) &\sim \Seq'(\cD^\gamma(\rho_\cD)) [z^k]\cD^\gamma(z) \notag \\
 &= \frac{1}{(1 - r)^2} \frac{L_k(1-r)^{\alpha+1}}{(1-s(r))r^\alpha} k^{-\alpha -1} \rho_\cD^{-k}.
\end{align*}
Here it holds that $\rho_\cD = \tau_\cD / \phi_\cD(\tau_\cD) = r(1-s(r))$. 

As $\nu_\cO<1$, it follows that $\tau_\cO = \rho_\cD = r(1-s(r))$.  The tree $\cT_n^\cO$ is distributed like a Galton--Watson tree conditioned on having $n$ vertices, with the offspring distribution $\xi(\cT^\cO)$ satisfying
\begin{align}
	\label{eq:mmref}
	\Pr{\xi(\cT^\cO) = k} &= \tau_\cO^k([z^k] \phi_\cO(z)) / \phi_\cO(\tau_\cO) \notag \\
						&\sim  \frac{L_k(1-r)^{\alpha}}{(1-s(r))r^\alpha} k^{-\alpha -1} .
\end{align}
Note that $\Ex{\xi^\cO} = \nu_\cO < 1$.

Having this asymptotic expansion at hand, we may argue similarly as in Lemma~\ref{le:diss} to show that $\mO_n^\omega$ (equipped with the uniform measure on its vertices) may be approximated in the Gromov--Hausdorff--Prokhorov sense by a large dissection. The convergence of this dissection towards a circle equipped with a uniform point then follows analogously as in Remark~\ref{eq:rem}, yielding convergence for $\mO_n^\omega$.

Let us make this explicit. Recall that the outerplanar map $\mO_n^\gamma$ corresponds via a combinatorial bijection to the enriched tree $(\cT_n^\cO, \beta_n^\cO)$. The family $\beta_n^\cO$ assigns to each vertex $v \in \cT_n^\cO$ an ordered sequence of dissections $\beta_n^\cO(v)$. Similarly as in the proof of Lemma~\ref{le:diss}, we let $u^* \in \cT_n^\cO$ denote the lexicographically first vertex with maximal outdegree. We let $A_n$ denote the tree obtained by cutting away all descendants of the vertex $u^*$ and let $(T_n^i)_{1 \le i \le d^+_{\cT_n^\cO}(u^*)}$ denote the ordered family of fringe subtrees dangling from $u^*$. 

It follows from~\cite[Thm. 20.1]{MR2908619} that the pruned tree $A_n$ converges in the local weak sense towards an almost surely finite tree. Hence its size is stochastically bounded. As for the fringe subtrees, Kortchemski's results~\cite[Thm. 1, Thm. 3]{MR3335012} imply that there is a slowly varying sequence $(\tilde{L}_n)_{n \ge 1}$ such that the sequence $C_n := \tilde{L}_n n^{1/\alpha}$ satisfies
\begin{align}
\label{eq:xx1}
d^+_{\cT_n^\cO}(u^*) = (1 - \nu_\cO)n + O_p(C_n) 
\end{align}
and such that the process $Z_k := |T_n^1| + \ldots + |T_n^k|$ satisfies
\begin{align}
\label{eq:xyo}
\max_{1 \le k \le d^+_{\cT_n^\cO}(u^*)} | Z_k - k /(1 - \nu_\cO)| = O_p(C_n).
\end{align}
Moreover, by a general Gibbs partition result~\cite[Thm. 3.1]{Mreplaceme} it follows that the largest dissection in the ordered  sequence $\beta_n^\cO(u^*)$ of dissections has size $d^+_{\cT_n^\cO}(u^*) + O_p(1)$.
Thus, the outerplanar map $\mO_n^\gamma$ consists of a giant dissection $D(\mO_n^\omega)$ of size 
\begin{align}
\label{eq:o1}
|D(\mO_n^\omega)| = (1 - \nu_\cO)n + O_p(C_n)
\end{align}
 such that each vertex $v_k$ of its counter-clockwise ordered vertices $v_1, \ldots, v_{|D(\mO_n^\omega|}$  is identified with the root-vertex of some outerplanar map $O_k$ with maximal size 
 \begin{align} 
 \label{eq:o2}
 \max_{1 \le i \le |D(\mO_n^\omega)|} |O_i| = O_p(C_n).
 \end{align} One of the $O_k$ corresponds to the union of $A_n$ and a stochastically bounded number of fringe subtrees dangling from of $u^*$, and each of the other $O_k$ corresponds to one of the remaining fringe subtrees dangling from ~$u^*$. It follows from~\eqref{eq:xyo} that
\begin{align}
\label{eq:o3}
\max_{1 \le k \le |D(\mO_n^\omega|} \left| |O_1| + \ldots |O_k| - k /(1 - \nu_\cO) \right| = O_p( C_n).
\end{align}
Let us say that each vertex $v \in \mO_n^\omega$ corresponds to the unique vertex $v_k \in D(\mO_n^\omega)$ with $v \in O_k$. It follows from \eqref{eq:o2} that the distortion of this correspondence has order $O_p(C_n)$. That is, the Hausdorff distance between the space $(\mO_n^\omega, d_{\mO_n^\omega})$ and the subspace $(D(\mO_n^\omega), d_{D(\mO_n^\omega)})$ lies in $O_p(C_n)$. Let $x \in \mO_n^\omega$ be drawn uniformly at random and let $v_{k(x)} \in \mD_n^\omega$ be its corresponding vertex on the dissection. It follows from \eqref{eq:o2} that the Prokhorov-distance between the distribution of $x$ and $v_{k(x)}$ lies in $O_p(C_n)$.  Hence, the distributions $\cL(x)$ and $\cL(v_{k(x)})$ of the random points of $x$ and $v_{k(x)}$ satisfy
\begin{align}
	\label{eq:ss1}
	d_{\mathrm{GHP}}\left( \left(\mO_n^\omega, \frac{1}{n(1- \nu_\cD)(1 - \nu_\cO)} d_{\mO_n^\omega}, \cL(x)\right) ,  \left(D(\mO_n^\omega), \frac{1}{n(1- \nu_\cD)}d_{D(\mO_n^\omega)}, \cL(v_{k(x)})\right) \right) \convdis 0.
\end{align}
For any $k\ge 1$ it holds that
\[
 	\Pr{ k(x) \le k} = \Exb{ |O_1| + \ldots + |O_k|} / n.
\]
By dominated convergence and Equations~\eqref{eq:o1} and \eqref{eq:o3} it follows that $\frac{k(x)}{n(1 - \nu_\cO)}$ converges weakly towards a uniform point of the unit interval $[0,1]$. Consequently, if $\mu_n'$ denotes the uniform measure on the vertices $D(\mO_n^\omega)$, it follows that
\begin{align}
\label{eq:ss2}
d_{\mathrm{GHP}}\left(  \left(D(\mO_n^\omega), \frac{1}{n(1- \nu_\cD)}d_{D(\mO_n^\omega)}, \cL(v_{k(x)})\right),   \left(D(\mO_n^\omega), \frac{1}{n(1- \nu_\cD)}d_{D(\mO_n^\omega)}, \mu_n')\right) \right) \convdis 0.
\end{align}
Let $\mu_n$ denote the uniform distribution on the vertices of $\mD_n^\gamma$, and $\mu$ the uniform distribution on the circle $C^1$. By identical arguments as for Remark~\ref{eq:rem} it follows that
\[
\left(\mD_n^\gamma, \frac{1}{ n(1-\nu_\cD)} d_{\mD_n^\gamma}, \mu_n \right) \convdis (C^1, d_{C^1}, \mu)
\]
in the Gromov--Hausdorff--Prokhorov sense. As $( D(\mO_n^\omega) \mid |D(\mO_n^\omega)| = k) \eqdist D_k(\mO_n^\omega)$ for any $k$, it follows from Equation \eqref{eq:o1} that 
\[
\left(D(\mO_n^\omega), \frac{1}{ n(1-\nu_\cD)} d_{D(\mO_n^\omega)}, \mu_n' \right) \convdis (C^1, d_{C^1}, \mu).
\]
Together with the limits \eqref{eq:ss1} and \eqref{eq:ss2} this implies that 
\[
\left(\mO_n^\omega, \frac{1}{n(1- \nu_\cD)(1 - \nu_\cO)} d_{\mO_n^\omega}, \cL(x)\right) \convdis (C^1, d_{C^1}, \mu).
\]
\end{proof}

\bibliographystyle{siam}
\bibliography{outerp}

\begin{thebibliography}{10}

\bibitem{MR1085326}
{\sc D.~Aldous}, {\em The continuum random tree. {I}}, Ann. Probab., 19 (1991),
  pp.~1--28.

\bibitem{MR1166406}
\leavevmode\vrule height 2pt depth -1.6pt width 23pt, {\em The continuum random
  tree. {II}. {A}n overview}, in Stochastic analysis ({D}urham, 1990), vol.~167
  of London Math. Soc. Lecture Note Ser., Cambridge Univ. Press, Cambridge,
  1991, pp.~23--70.

\bibitem{MR1207226}
\leavevmode\vrule height 2pt depth -1.6pt width 23pt, {\em The continuum random
  tree. {III}}, Ann. Probab., 21 (1993), pp.~248--289.

\bibitem{MR1406564}
{\sc J.~Bertoin}, {\em L\'evy processes}, vol.~121 of Cambridge Tracts in
  Mathematics, Cambridge University Press, Cambridge, 1996.

\bibitem{MR2185278}
{\sc N.~Bonichon, C.~Gavoille, and N.~Hanusse}, {\em Canonical decomposition of
  outerplanar maps and application to enumeration, coding and generation}, J.
  Graph Algorithms Appl., 9 (2005), pp.~185--204 (electronic).

\bibitem{MR3484733}
{\sc T.~Budd}, {\em The peeling process of infinite {B}oltzmann planar maps},
  Electron. J. Combin., 23 (2016), pp.~Paper 1.28, 37.

\bibitem{MR3646061}
{\sc T.~Budd and N.~Curien}, {\em Geometry of infinite planar maps with high
  degrees}, Electron. J. Probab., 22 (2017), pp.~Paper No. 35, 37.

\bibitem{caraceni2016}
{\sc A.~Caraceni}, {\em The scaling limit of random outerplanar maps}, Ann.
  Inst. H. Poincar\'e Probab. Statist., 52 (2016), pp.~1667--1686.

\bibitem{MR3382675}
{\sc N.~Curien, B.~Haas, and I.~Kortchemski}, {\em The {CRT} is the scaling
  limit of random dissections}, Random Structures Algorithms, 47 (2015),
  pp.~304--327.

\bibitem{MR3286462}
{\sc N.~Curien and I.~Kortchemski}, {\em Random stable looptrees}, Electron. J.
  Probab., 19 (2014), pp.~no. 108, 35.

\bibitem{MR3405619}
\leavevmode\vrule height 2pt depth -1.6pt width 23pt, {\em Percolation on
  random triangulations and stable looptrees}, Probab. Theory Related Fields,
  163 (2015), pp.~303--337.

\bibitem{MR1964956}
{\sc T.~Duquesne}, {\em A limit theorem for the contour process of conditioned
  {G}alton-{W}atson trees}, Ann. Probab., 31 (2003), pp.~996--1027.

\bibitem{MR1954248}
{\sc T.~Duquesne and J.-F. Le~Gall}, {\em Random trees, {L}\'evy processes and
  spatial branching processes}, Ast\'erisque,  (2002), pp.~vi+147.

\bibitem{MR0270403}
{\sc W.~Feller}, {\em An introduction to probability theory and its
  applications. {V}ol. {II}.}, Second edition, John Wiley \& Sons, Inc., New
  York-London-Sydney, 1971.

\bibitem{MR3097424}
{\sc S.~Foss, D.~Korshunov, and S.~Zachary}, {\em An introduction to
  heavy-tailed and subexponential distributions}, Springer Series in Operations
  Research and Financial Engineering, Springer, New York, second~ed., 2013.

\bibitem{MR3650252}
{\sc I.~Geffner and M.~Noy}, {\em Counting outerplanar maps}, Electron. J.
  Combin., 24 (2017), pp.~Paper 2.3, 8.

\bibitem{MR3050512}
{\sc B.~Haas and G.~Miermont}, {\em Scaling limits of {M}arkov branching trees
  with applications to {G}alton-{W}atson and random unordered trees}, Ann.
  Probab., 40 (2012), pp.~2589--2666.

\bibitem{MR959133}
{\sc J.~Jacod and A.~N. Shiryaev}, {\em Limit theorems for stochastic
  processes}, vol.~288 of Grundlehren der Mathematischen Wissenschaften
  [Fundamental Principles of Mathematical Sciences], Springer-Verlag, Berlin,
  1987.

\bibitem{MR2908619}
{\sc S.~Janson}, {\em Simply generated trees, conditioned {G}alton-{W}atson
  trees, random allocations and condensation}, Probab. Surv., 9 (2012),
  pp.~103--252.

\bibitem{MR2860856}
{\sc S.~Janson, T.~Jonsson, and S.~{\"O}. Stef{\'a}nsson}, {\em Random trees
  with superexponential branching weights}, J. Phys. A, 44 (2011), pp.~485002,
  16.

\bibitem{MR3342658}
{\sc S.~Janson and S.~{\"O}. Stef{\'a}nsson}, {\em Scaling limits of random
  planar maps with a unique large face}, Ann. Probab., 43 (2015),
  pp.~1045--1081.

\bibitem{MR2764126}
{\sc T.~Jonsson and S.~{\"O}. Stef{\'a}nsson}, {\em Condensation in nongeneric
  trees}, J. Stat. Phys., 142 (2011), pp.~277--313.

\bibitem{MR871905}
{\sc H.~Kesten}, {\em Subdiffusive behavior of random walk on a random
  cluster}, Ann. Inst. H. Poincar\'e Probab. Statist., 22 (1986), pp.~425--487.

\bibitem{MR2946438}
{\sc I.~Kortchemski}, {\em Invariance principles for {G}alton-{W}atson trees
  conditioned on the number of leaves}, Stochastic Process. Appl., 122 (2012),
  pp.~3126--3172.

\bibitem{MR3185928}
\leavevmode\vrule height 2pt depth -1.6pt width 23pt, {\em A simple proof of
  {D}uquesne's theorem on contour processes of conditioned {G}alton-{W}atson
  trees}, in S\'eminaire de {P}robabilit\'es {XLV}, vol.~2078 of Lecture Notes
  in Math., Springer, Cham, 2013, pp.~537--558.

\bibitem{MR3335012}
\leavevmode\vrule height 2pt depth -1.6pt width 23pt, {\em Limit theorems for
  conditioned non-generic {G}alton-{W}atson trees}, Ann. Inst. Henri Poincar\'e
  Probab. Stat., 51 (2015), pp.~489--511.

\bibitem{2015arXiv150404358K}
{\sc I.~{Kortchemski}}, {\em {Sub-exponential tail bounds for conditioned
  stable Bienaym$\backslash$'e-Galton-Watson trees}}, ArXiv e-prints,  (2015).

\bibitem{MR2778796}
{\sc J.-F. Le~Gall and G.~Miermont}, {\em Scaling limits of random planar maps
  with large faces}, Ann. Probab., 39 (2011), pp.~1--69.

\bibitem{MR3025391}
\leavevmode\vrule height 2pt depth -1.6pt width 23pt, {\em Scaling limits of
  random trees and planar maps}, in Probability and statistical physics in two
  and more dimensions, vol.~15 of Clay Math. Proc., Amer. Math. Soc.,
  Providence, RI, 2012, pp.~155--211.

\bibitem{2016arXiv161208618M}
{\sc C.~{Marzouk}}, {\em {Scaling limits of random bipartite planar maps with a
  prescribed degree sequence}}, ArXiv e-prints,  (2016).

\bibitem{MR2571957}
{\sc G.~Miermont}, {\em Tessellations of random maps of arbitrary genus}, Ann.
  Sci. \'Ec. Norm. Sup\'er. (4), 42 (2009), pp.~725--781.

\bibitem{MR2245368}
{\sc J.~Pitman}, {\em Combinatorial stochastic processes}, vol.~1875 of Lecture
  Notes in Mathematics, Springer-Verlag, Berlin, 2006.
\newblock Lectures from the 32nd Summer School on Probability Theory held in
  Saint-Flour, July 7--24, 2002, With a foreword by Jean Picard.

\bibitem{2017arXiv170401950R}
{\sc L.~{Richier}}, {\em {Limits of the boundary of random planar maps}}, ArXiv
  e-prints,  (2017).

\bibitem{Mreplaceme}
{\sc B.~{Stufler}}, {\em {Gibbs partitions: the convergent case}}, ArXiv
  e-prints,  (2016).

\bibitem{2016arXiv161202580S}
\leavevmode\vrule height 2pt depth -1.6pt width 23pt, {\em {Limits of random
  tree-like discrete structures}}, ArXiv e-prints,  (2016).

\bibitem{stufler2017}
{\sc B.~Stufler}, {\em Scaling limits of random outerplanar maps with
  independent link-weights}, Ann. Inst. H. Poincar\'e Probab. Statist., 53
  (2017), pp.~900--915.

\end{thebibliography}

\end{document}